\documentclass[11pt]{amsart}
\usepackage[utf8]{inputenc}
\usepackage[numbered]{bookmark}
\usepackage[english]{babel}
\usepackage{amssymb}
\usepackage{amsmath}
\usepackage{amsthm}
\usepackage{dsfont}
\allowdisplaybreaks
\usepackage{kotex}
\usepackage{ulem}
\usepackage[left=30mm, right=30mm, top=30mm, bottom=30mm]{geometry}
\usepackage{tikzit}
\usepackage{tikzscale}
\tikzstyle{blue circle opacity 0.5}=[
scale=1.2,
pattern = north west lines, pattern color=blue, 
draw=black, shape=circle]
\tikzstyle{new style 1}=[
scale=1.2, 
fill = red, draw=black, shape=circle, opacity=0.5]

\tikzstyle{dashed edge}=[dashed]
\tikzstyle{thick black arrow}=[draw=black, ->, thick]

\usetikzlibrary{patterns}

\usepackage[capitalise]{cleveref} 

\usepackage{amsfonts}
\usepackage[numbers]{natbib}
\usepackage[shortlabels]{enumitem}
\usepackage{mathtools}

\newcommand{\ngcd}{\operatorname{ngcd}}
\newcommand{\tp}{\mathsf{T}}

\usepackage{color}

\title[On the common zeros of quasi-modular forms for \(\Gamma_0^+(N)\) of level $N=1,2,3$]{On the common zeros of quasi-modular forms for \(\Gamma_0^+(N)\)\\ of level $N=1,2,3$}

\author{Bo-Hae Im}
\address{Department of Mathematical Sciences, KAIST, 291 Daehak-ro, Yuseong-gu, Daejeon, 34141, South Korea
} \email{bhim@kaist.ac.kr}

\author{Hojin Kim}
\address{Department of Mathematical Sciences, KAIST, 291 Daehak-ro, Yuseong-gu, Daejeon, 34141, South Korea
} \email{hojinkim@kaist.ac.kr}

\author{Wonwoong Lee}
\address{Department of Mathematical Sciences, KAIST, 291 Daehak-ro, Yuseong-gu, Daejeon, 34141, South Korea
} \email{leeww@kaist.ac.kr}

\subjclass[2020]{Primary 11F11, 11F99}
\keywords{Modular forms, Quasi-modular forms, Fricke group}
\thanks{This work was supported by the National Research Foundation of Korea(NRF) grant funded by the Korea government(MSIT)
(No.~2020R1A2B5B01001835).}

\theoremstyle{definition}

\newtheorem{theorem}{Theorem}[section]
\newtheorem{definition}[theorem]{Definition}

\newtheorem{lem}[theorem]{Lemma}
\newtheorem{prop}[theorem]{Proposition}
\newtheorem{cor}[theorem]{Corollary}
\newtheorem*{rmk}{Remark}
\newtheorem{example}[theorem]{Example}

\newcommand{\Et}{E_2^{(2)}}
\newcommand{\Ef}{E_4^{(2)}}
\newcommand{\Es}{E_6^{(2)}}
\newcommand{\Ett}{E_2^{(3)}}
\newcommand{\Eff}{E_4^{(3)}}
\newcommand{\Ess}{E_6^{(3)}}
\newcommand{\Etn}{E_2^{(N)}}
\newcommand{\Efn}{E_4^{(N)}}
\newcommand{\Esn}{E_6^{(N)}}

\newcommand{\SL}{\mathrm{SL}}
\newcommand{\Z}{\mathbb{Z}}
\newcommand{\Q}{\mathbb{Q}}
\newcommand{\Ha}{\mathbb{H}}
\newcommand{\C}{\mathbb{C}}

\newcommand{\Zc}{\mathcal{Z}}
\newcommand{\Ac}{\mathds{A}}

\newcommand{\wtdeg}{\operatorname{wtdeg}}
\newcommand{\px}{\frac{\partial}{\partial x}}

\newcommand{\QQ}{\mathbb{Q}}

 \vspace{-8ex}
  \date{\today}
\begin{document}

\maketitle

\begin{abstract}
    In this paper, we study common zeros of the iterated derivatives of the Eisenstein series for $\Gamma_0^+(N)$ of level $N=1,2$ and $3$, which are  quasi-modular forms. More precisely, we  investigate the common zeros of quasi-modular forms, and prove that all the zeros of the iterated derivatives of the Eisenstein series $\frac{d^m E_k^{(N)}(\tau)}{d\tau^m}$ of weight $k=2,4,6$ for $\Gamma_0^+(N)$ of level $N=2,3$ are simple by generalizaing the results of  Meher \cite{MEH} and  Gun and Oesterl\'{e} \cite{SJ20} for $\SL_2(\Z)$.
\end{abstract}

\section{Introduction and the main results}

 Let \(\Gamma\) denote a Fuchsian group of the first kind, and for a positive integer $N$, let \(\Gamma_0^+(N)\) be a subgroup of \(\mathrm{SL}_2(\mathbb R)\) generated by the Hecke congruence group \(\Gamma_0(N)\) and the Fricke involution \(w_N:=
\begin{pmatrix}
0 & -\frac{1}{\sqrt N} \\
\sqrt N & 0 
\end{pmatrix}\), and \(\mathbb H\) be the complex upper half-plane.

First, we recall the definition of a quasi-modular form for $\Gamma$ which Kaneko and Zagier \cite{4} have introduced.
\begin{definition}\label{def1.1}
For a positive integer \(k\) and a non-negative $\ell$, a quasi-modular form of weight \(k\) and depth \(\ell\) for \(\Gamma\) is a holomorphic function \(f\) on \(\mathbb H\) satisfying the following conditions:
\begin{enumerate}[(i)]
    \item There exist holomorphic functions \(Q_i(f)\) for \(i=0,1,\ldots,\ell\) that satisfy
\begin{equation*}
    f[\gamma]_k=\sum_{i=0}^{\ell} Q_i(f)X(\gamma)^i, \quad \text{ with }  Q_{\ell}(f) \not\equiv 0, \text{ for all } \gamma=\begin{pmatrix}
a & b \\
c & d
\end{pmatrix} \in \Gamma,
\end{equation*}
 where the operator \([\gamma]_k\) is defined by
\begin{equation*}
    f[\gamma]_k(\tau)=(c\tau+d)^{-k}f(\gamma \tau),
\end{equation*}
and the function \(X(\gamma)\) is defined by
\begin{equation*}
    X(\gamma)(\tau)=\frac{\tau}{c\tau+d}.
\end{equation*}
    \item \(f\) is polynomially bounded, i.e., there exists a constant \(\alpha>0\) such that \(f(\tau)=O((1+|\tau|^2)/v)^{\alpha}\) as \(v \to \infty\) and \(v \to 0\), where \(v=\Im(\tau)\).
\end{enumerate}
\end{definition}

One may define the quasi-modular form using the notion of almost holomorphic modular forms, that is, the constant term $F_0(\tau)$ in the variable $v$ of an almost holomorphic modular form $F(z):=\sum_{i=0}^l F_i(\tau)(-4 \pi v)^{-i}.$ In fact, the notions of quasi-modular forms and almost holomorphic modular forms coincide since if $f=F_0$ is a quasi-modular form inherited from an almost holomorphic modular form $F(z):=\sum_{i=0}^l F_i(\tau)(-4 \pi v)^{-i}$ then $Q_i(f)=F_i.$ See \cite[Section 5.3]{BRU} for more details.

\

The following proposition shows the structure of the space of quasi-modular forms.

\begin{prop}\label{prop1.7}{\cite[Proposition 20]{BRU}} Let $\ell$ be a non-negative integer and \(\Gamma\) be a non-cocompact Fuchsian group of the first kind. For a non-negative integer $k$, denote by \(M_k(\Gamma)\) the space of modular forms of weight~\(k\) for \(\Gamma\) and by \(\widetilde{M}_k^{(\leq \ell)}(\Gamma)\) the space of quasi-modular forms of weight \(k\) and depth \(\leq \ell\) for \(\Gamma\). Let $\theta:=\frac{1}{2\pi i}\frac{d}{d\tau}$ and \(\phi\) be a quasi-modular form of weight 2 for \(\Gamma\) which is not modular. Then we have the following:
\begin{enumerate}[(i)]
    \item \(\theta(\widetilde{M}_k^{(\leq \ell)}(\Gamma)) \subseteq \widetilde{M}_{k+2}^{(\leq \ell+1)}(\Gamma)\).
    \item \(\widetilde{M}_k^{(\leq \ell)}(\Gamma)=\bigoplus_{r=0}^{\ell} M_{k-2r}(\Gamma)\cdot \phi^r\).
\end{enumerate}
\end{prop}

\

Let \(j_2\) and \(j_3\) be the Hauptmodul for \(\Gamma_0^+(2)\) and \(\Gamma_0^+(3)\), respectively, defined by 
\begin{align*}
j_2(\tau)&=\left(\frac{\eta(\tau)}{\eta(2\tau)}\right)^{24}+24+2^{12}\left(\frac{\eta(2\tau)}{\eta(\tau)}\right)^{24}, \\
j_3(\tau)&=\left(\frac{\eta(\tau)}{\eta(3\tau)}\right)^{12}+12+3^{6}\left(\frac{\eta(3\tau)}{\eta(\tau)}\right)^{12},
\end{align*}
where $\eta(\tau)$ is the Dedekind eta function,
\begin{equation*}\label{notation}
\begin{aligned}
&\eta(\tau)=q^{1/24}\prod_{n=1}^{\infty}(1-q^n).
\end{aligned}
\end{equation*}

For each positive even  integer \(k\), let
\begin{equation*}
\begin{aligned}
&E_{k}^{(N)}(\tau)=\frac{1}{1+N^{k/2}}(E_{k}(\tau)+N^{k/2}E_{k}(N\tau)), \quad \text{for } N=1,2,3.
\end{aligned}
\end{equation*}

We note that $E_k^{(N)}$ is a modular form of weight $k$ for $\Gamma_0^+(N)$ if $k \geq 4$, and in particular for $N=1$,
$$E_k^{(1)}(\tau):=E_k(\tau)=1-\frac{2k}{B_k}\sum_{n=1}^{\infty}\sigma _{k-1}(n)\tau^n,\quad 
\text{ where } \sigma_{k-1}(n)=\sum_{d\mid n}d^{k-1},$$ is the standard Eisenstein series of weight $k$ for $\SL_2(\Z)$.  One can express \(j_N\) as a rational function in \(\Efn\) and \(\Esn\)(see~\cite{JST}).

For \(N=2,3\) and for a non-negative integer $\ell \in \Z_{\geq 0}$,  Proposition \ref{prop1.7} implies that every quasi-modular form \(f\) of weight \(k\) and depth \(\leq \ell\) for \(\Gamma_0^+(N)\) can be written as
\begin{equation*}
    f(\tau)=f_0(\tau)+f_1(\tau)E_2^{(N)}(\tau)+\cdots+f_p(\tau)(E_2^{(N)}(\tau))^{\ell},
\end{equation*}
where for \(0 \leq i \leq \ell\),  \(f_i\) is a modular form of weight \(k-2i\) for $\Gamma_0(N)$.

\

The zeros of certain modular forms for \(\mathrm{SL}_2(\mathbb Z\)) have been studied actively. It dates back to Rankin and Swinnerton-Dyer's celebrated result \cite{RS70} which has proved Wohlfahrt's conjecture~\cite{Woh64}, that is, they have proved that all the zeros of the Eisenstein series $E_k$ for $k \geq 4$ lies on the unit circle $|\tau|=1$ in the stardard fundamental domain $\{\tau \in \Ha : -\frac{1}{2} \leq \Re{\tau} \leq \frac{1}{2}, \text{ } |\tau| \geq 1\}$. Since then, extensive and various studies on the zeros of modular and quasi-modular forms have been conducted in this regard. Getz \cite{7} has generalized Rankin and Swinnerton-Dyer's result \cite{RS70} for the modular form of arbitrary weight $k \geq 4$ for $\SL_2(\Z).$ They proved that if a holomorphic modular form $f(\tau)$ has the Fourier expansion of the form $f(\tau)=1+cq^{\dim S_k+1}+\cdots$ for some $c\in \C$, where $q:=e^{2\pi i \tau}$ and $S_k$ is the space of cusp forms of weight $k$ for $\SL_2(\Z)$, then the zeros of $f(z)$ has the same property.  On the other hand, Rankin \cite{Ran82} has extended Rankin and Swinnerton's results \cite{RS70} to the Poincare series whose order is a rational function with real coefficients. Gun \cite{8} has generalized the argument of Rankin \cite{Ran82} to prove a similar result of \cite{7} for certain cusp forms. Kohnen \cite{9} has given the explicit formula of zeros of $E_2$ lying on the unit circle in the right half plane.

Also, for the aspect of our interest, we note that Saber and Sebbar \cite{17} have studied the zeros of \(f'\) of a modular form \(f\) for \(\mathrm{SL}_2(\mathbb Z)\). For an even integer \(k \geq 4\), Balasubramanina and Gun \cite{11} have studied the transcendence of the zeros of \(\theta E_k(\tau)\), and Miezaki, Nozaki and Shigezumi \cite{12} have studied the location of the zeros of the Eisenstein series $E_k^{(N)}$ for \(\Gamma_0^+(N)\), for $N=2,3$.

\


In \cite{Woh64}, together with the famous conjecture resolved in \cite{RS70}, Wohlfahrt has proved that all  the zeros of the Eisenstein series $E_k$ for weight $4 \leq k \leq 26$ are simple. El Basraoui and Sebbar \cite{ES10} have paid attention to the Eisenstein series $E_2$ instead of the Eisenstein series $E_k$ for $k \geq 4$ which are modular forms. They showed that there are infinitely many $\SL_2(\Z)$-inequivalent zeros of $E_2$, and they also proved that the multiplicity of the zero of $E_2$ is also simple. 

On the other hand, Choi and the first named author in \cite{6} and \cite{13} have traced the location of the Eisenstein series of weight 2 for the Fricke group $\Gamma_0^+(N).$ In particular, they have proved that the Eisenstein series $E_2^{(N)}$ has infinitely many $\Gamma_0^+(N)$-inequivalent zeros in the strip $\{z \in \Ha :- \frac{1}{2} < \Re{z} \leq \frac{1}{2} \}$ and for $N=2$ there is a fundamental domain of $\Gamma_0^+(N)$ in which $E_2^{(N)}$ does not have zeros. They have also proved that there is exactly one $\Gamma_0^+(N)$-inequivalent zero lying on each of the mixed Ford circles for $N=2,3.$ 

\

Meher \cite{MEH} has focused on the multiplicity of zeros. More specifically, Meher has studied several properties on the common zeros of quasi-modular forms for $\SL_2(\Z)$, and then has proved that all the zeros of $E_2, \frac{dE_2}{d\tau}, \frac{d^2E_2}{d\tau^2}$ and $\frac{d^3E_2}{d\tau^3}$ are simple, which is a generalization of the result in \cite{ES10}. Then, Gun and Oesterl\'{e} \cite{SJ20} have generalized Meher's result. More precisely, they have provided certain conditions which imply that two quasi-modular forms for $\SL_2(\Z)$ have no common zero and proved that for $k \geq 2$, the iterated derivative  $\frac{d ^j E_k}{d\tau^j}$ of the Eisenstein series for an arbitrary order $j \geq 0$ has only simple zeros.

\

In this paper, we generalize this result for the Fricke groups $\Gamma_0^+(2)$ and $\Gamma_0^+(3)$ and prove the following result.

\begin{theorem}
\label{thm: simple zero quasimodular} For $N=2,3$, all zeros of $\dfrac{d^j E_2^{(N)}(\tau)}{d\tau^j}$, $\dfrac{d^jE_4^{(N)}(\tau)}{d\tau^j}$, $\dfrac{d^jE_6^{(N)}(\tau)}{d\tau^j}$ are simple for all integers $j\ge0$. 
\end{theorem}

We give the proof of Theorem~\ref{thm: simple zero quasimodular} in Section~\ref{sec: simplicity} and Appendix~\ref{sec: appendix}. The idea of the proof is based on the method of Gun and Oeterl\'{e} given in \cite{SJ20} which states that the problem on the common zeros of quasi-modular forms may be reduced to the problem on certain polynomials and their common divisors. However, contrary to the case $N=1$, the {\it Ramanujan identities} for $N=2,3$ (see \eqref{Ramanujan2} and \eqref{Ramanujan3} in Section \ref{sec: common zero}) are not of the polynomial forms anymore. This is a serious obstruction to resolving the problem.  In Section~\ref{sec: simplicity}, we show how to overcome this obstruction by a delicate analysis of the arithmetic nature of coefficients that appear in certain polynomials associated to the iterated derivatives of the Eisenstein series for $\Gamma_0^+(N)$ for $N=2,3$.  We note that some calculations for $\Gamma_0^+(3)$ is postponed to Appendix~\ref{sec: appendix}, since the main idea is same with the case for $\Gamma_0^+(2)$, except that the required calculations are more complicated for $\Gamma_0^+(3)$. We also treat the case for $N=1$ for reader's convinience.


In Section~\ref{sec: common zero}, we investigate the common zeros of various quasi-modular forms for $\Gamma_0^+(N)$ of level $N=1,2,3.$ Especially, we prove that there are no common zeros of $E_2^{(N)}$ and $\frac{d^m E_2^{(N)}(\tau)}{d\tau ^m}$ for any positive integer $m$. This can be shown by two different methods (see Proposition \ref{prop_ratandE2} and Proposition \ref{prop_f'andE2}). We also prove that if $j_N(\alpha)$ is an algebraic number, then $\alpha$ can never be a zero of $\frac{d^n E_2^{(N)}(\tau)}{d\tau ^n}$ for $n \geq 0$ (see Proposition \ref{prop1.3}). These are  the generalization of all of Meher's results in \cite{MEH}, and the proofs are similar. 
We also give various pairs of quasi-modular forms for $\SL_2(\Z)$ which have no common zero. We show that there is no common zero of any holomorphic modular form and a quasi-modular form of maximal depth. Also we deduce that there is no common zero of $\frac{d^m E_k}{d\tau ^m}$ and arbitrary non-zero holomorphic modular form with rational Fourier coefficients for $\SL_2(\Z)$, for weight $k=4$ and $6$.

Also, as an extension of Meher's result \cite[Theorem 2.4]{MEH},  we prove that there are infinitely many $\Gamma_0^+(N)$-inequivalent zeros in $\Ha$ of a quasi-modular form of depth $1$ in Theorem \ref{thm1.2} of Section \ref{sec2}.

\section{Equivariant forms}\label{sec2}

El Basraoui and Sebbar have investigated several properties of equivariant forms in \cite{equi-s}. Let us recall them briefly, first.

The ``double-slash" operator is defined for a meromorphic function \(f\) on \(\mathbb H\) by
\begin{equation*}
    f\|[\gamma](\tau)=(c\tau+d)^{-2}f(\gamma \tau)-\frac{c}{(c\tau+d)}, \quad \text{for } \gamma=\begin{pmatrix}
    a & b \\
    c & d
    \end{pmatrix} \in \mathrm{SL}_2(\mathbb R).
\end{equation*}

For a function \(h:\mathbb H \to \mathbb C\) which is not the identity, we let
\begin{equation*}
    \hat{h}(\tau):=\frac{1}{h(\tau)-\tau}.
\end{equation*}

 A meromorphic function \(h\) on $\Ha$ is called {\it an equivariant function} for \(\Gamma\) if it satisfies
\begin{equation*}
    h(\gamma \tau)=\gamma h(\tau), \quad \text{for each } \gamma \in \Gamma.
\end{equation*}
If \(\hat h\) is meromorphic at every cusp of \(\Gamma\), then \(\hat h\) is called {\it an equivariant form} for \(\Gamma\). Here we say that \(h\) is meromorphic at the cusp \(s\) if \(\hat h \|[\delta](\tau)\) is meromorphic at \(\infty\), where \(\delta \in \mathrm{SL}_2 (\mathbb R)\) sends \(\infty\) to \(s\).

\begin{prop}{\cite[Proposition 3.1]{equi-s}}\label{equivariant}
Let \(h\) be a meromorphic function on \(\mathbb H\). Then, for \(\gamma \in \Gamma\) and \(\tau \in \mathbb H\), we have 
\begin{equation*}
    h(\gamma \tau)=\gamma h(\tau) \text{ if and only if } \hat h \|[\gamma](\tau)=\hat h (\tau). 
\end{equation*}
\end{prop}

\begin{lem}\label{lem_slash23}
If \(f=f_0+f_1E_2^{(2)}\) is a quasi-modular form of weight \(k\) and depth~\(1\) for \(\Gamma_0^+(2)\), then for \(\tau \in \mathbb H\), \(h_2(\tau):=\tau+\frac{4}{i\pi}\frac{f_1(\tau)}{f(\tau)}\) is an equivariant form for \(\Gamma_0^+(2)\). Similarly, if \(g=g_0+g_1E_2^{(3)}\) is a quasi-modular form of weight \(k\) and depth~$1$ for \(\Gamma_0^+(3)\), then \(h_3(\tau)=\tau+\frac{3}{i\pi}\frac{g_1(\tau)}{g(\tau)}\) is an equivariant form for \(\Gamma_0^+(3)\). 
\end{lem}

\begin{proof}
Here \(\hat {h}_2(\tau)=\frac{i\pi}{4}\frac{f(\tau)}{f_1(\tau)}\). Let \(\gamma \in \Gamma_0^+(2)\). Applying the double-slash operator \( \|[\gamma]\) on $\hat{h}_2$ and using the fact that
\begin{equation}
    E_2^{(2)}(\gamma \tau)=(c\tau+d)^2E_2^{(2)}(\tau)+\frac{4}{i\pi}c(c\tau+d) \quad  \text{ for } \gamma=\begin{pmatrix}
    a & b \\
    c & d
    \end{pmatrix} \in \mathrm{SL}_2(\mathbb Z) \cup \Gamma_0^+(2) \label{eq_1},
\end{equation}
we have
\begin{align*}
    \hat {h}_2\|[\gamma](\tau)&=(c\tau+d)^{-2}\hat h_2(\gamma \tau)-c(c\tau+d)^{-1} \\
    &=(c\tau+d)^{-2}\left(\frac{i\pi}{4}\frac{f_0(\gamma \tau)}{f_1(\gamma \tau)}+\frac{i\pi}{4}E_2^{(2)}(\gamma \tau)\right)-c(c\tau+d)^{-1} \\
    &=\frac{i\pi}{4}\frac{f_0(\tau)}{f_1(\tau)}+\frac{i\pi}{4}E_2^{(2)}(\tau)
    =\hat {h_2}(\tau).
\end{align*}
By Proposition \ref{equivariant}, \(h_2(\gamma \tau)=\gamma h_2(\tau)\) for \(\gamma \in \Gamma_0^+(2)\), i.e., \(h_2\) is an equivariant function for $\Gamma_0^+(2)$. In the same manner, since
\begin{equation}
    E_2^{(3)}(\gamma \tau)=(c\tau+d)^2E_2^{(3)}(\tau)+\frac{3}{i\pi}c(c\tau+d) \text{ for } \gamma=\begin{pmatrix}
    a & b \\
    c & d
    \end{pmatrix} \in \mathrm{SL}_2(\mathbb Z) \cup \Gamma_0^+(3) \label{eq_1(3)},
\end{equation}
we also see that \(h_3\) is an equivariant function for $\Gamma_0^+(3)$.

Let \(\delta \in \mathrm{SL}_2(\mathbb R)\) send \(\infty\) to the cusp \(s\) for \(\Gamma_0^+(2)\). We see that
\begin{align*}
    \hat h_2\|[\delta](\tau) &=\frac{i\pi}{4}\frac{f[\delta]_k(\tau)}{f_1[\delta]_{k-2}(\tau)}-c(c\tau+d)^{-1} \\
    &=\frac{i\pi}{4}\frac{f_0[\delta]_k(\tau)+f_1[\delta]_{k-2}(\tau)E_2^{(2)}[\delta]_2(\tau)}{f_1[\delta]_{k-2}(\tau)}-c(c\tau+d)^{-1} \\
    &=\frac{i\pi}{4}\frac{f_0[\delta]_k(\tau)}{f_1[\delta]_{k-2}(\tau)}+\frac{i\pi}{4}E_2^{(2)}[\delta](\tau)-c(c\tau+d)^{-1} \\
    &=\frac{i\pi}{4}\left(\frac{f_0[\delta]_k(\tau)}{f_1[\delta]_{k-2}(\tau)}+E_2^{(2)}(\tau)\right),
\end{align*}
where the last equality follows from \eqref{eq_1}. Note that since the function \(f_0\) and \(f_1\) are modular forms of weight $k$ and $k-2$, respectively, for \(\Gamma_0^+(2)\), they are holomorphic at the cusp \(s\), i.e., \(f_0[\delta]_k\) and \(f_1[\delta]_{k-2}\) are holomorphic at \(\infty\). Therefore, \(\hat h_2\) is meromorphic at the cusp \(s\), so \(h\) is an equivariant form for \(\Gamma_0^+(2)\). The same argument holds for \(h_3\).
\end{proof}

\begin{theorem}\label{thm1.2}
Let \(N=2,3\). Then for  a quasi-modular form \(f\) of weight \(k\) and depth~$1$  for \(\Gamma_0^+(N)\), there are infinitely many \(\Gamma_0^+(N)\)-inequivalent zeros of \(f\) in \(\mathbb H\).
\end{theorem}

\begin{proof}
We follow the same argument as in the proof of \cite[Theorem 4.1]{17}, and we give the details of the proof for \(N=2\) for readers. We note that the case \(N=3\) can be proved by applying the same argument. Let $f(\tau)=f_0(\tau)+f_1(\tau)E_2^{(2)}$ as in Lemma \ref{lem_slash23} and let \(h(\tau)=\tau+\frac{4}{i\pi}\frac{f_1(\tau)}{f(\tau)}\) be an equivariant form for \(\Gamma_0^+(2)\). There is a point $\tau_0$ $\in \Ha$? with $h(\tau_0) \in \mathbb R$. Indeed, if \(h\) has a pole, say \(\tau_1\), then 
\begin{equation*}
    h(\gamma \tau)=\gamma h(\tau)=\gamma \infty \in \mathbb Q
\end{equation*}
for some \(\gamma \in \Gamma_0^+(2)\). If \(h\) is a equivariant form that is holomorphic in \(\mathbb H\), then by \cite[Proposition 3.3, 3.4]{equi-s}, \(h(\mathbb H)\) is not contained in \(\mathbb H\), also not contained in \(\mathbb H^-\), so \(h(\mathbb H) \cap \mathbb R \not= \emptyset\). In either case, there is a \(\tau_0 \in \mathbb H\) with \(h(\tau_0) \in \mathbb R\), and it is not an elliptic point; otherwise, \(h(\tau_0)\) is equal to \(\tau_0\) or \(\overline{\tau_0}\), both are not in \(\mathbb R\).

Let $V$ be an open set of $\tau_0$ such that there are no elliptic points in \(V\), and $U$ be an open set of $\tau_0$ such that there are no poles of $h$ in $U$. Let $D=h(V \cap U)$, which is open in $H$. Since each orbit of a cusp for $\Gamma_0(2)$ is dense in $\mathbb{R}$, there are infinitely many rational points that are $\Gamma_0(2)$-equivalent to the cusp $\infty$. Let $r$ be a such rational number with $r=h(\tau)$ where $\tau \in U \cap V$ and let $\gamma \in \Gamma_0(2)$ such that $\gamma r = \infty$. Then $\gamma \tau$ is a pole of $h$ and we can take infinitely many such $r$ and $\tau$, and they are all $\Gamma_0^+(2)$-inequivalent. Therefore the function $f$ has infinitely many inequivalent zeros for $\Gamma_0^+(2)$ in $\mathbb H$.
\end{proof}

\section{\texorpdfstring{Common zeros of quasi-modular forms for $N=1,2,3$}{Common zeros of quasi-modular forms for N=1,2,3}}\label{sec: common zero}

\subsection{Common zeros for $N=2,3$}

In 1916, Ramanujan proved the following {\it Ramanujan identity} for $\SL_2(\Z)$:
\begin{equation}\label{rama}
\begin{cases}
    \theta E_2=\frac{1}{12}(E_2^2-E_4), \\
    \theta E_4=\frac{1}{3}(E_2E_4-E_6), \\
    \theta E_6=\frac{1}{2}(E_2E_6-E_4^2).
\end{cases}
\end{equation}
Zudilin \cite{ZUD} showed the analogues of the Ramanujan identities for \(\Gamma_0^+(2)\) and \(\Gamma_0^+(3)\), that is,
\begin{equation}\label{Ramanujan2}
\begin{cases}
    \theta E_2^{(2)}=\frac{1}{8}\left((\Et)^2-\Ef\right), \\
    \theta E_4^{(2)}=\frac{1}{2}\left(\Et \Ef - \Es\right), \\
    \theta E_6^{(2)}=\frac{1}{4}\left(3\Et \Es-2(\Ef)^{2}-\frac{(\Es)^2}{\Ef}\right),
\end{cases}
\end{equation}
and
\begin{equation}\label{Ramanujan3}
    \begin{cases}
    \theta E_2^{(3)}=\frac{1}{6}\left((\Ett)^2-\Eff\right), \\
    \theta E_4^{(3)}=\frac{2}{3}\left(\Ett \Eff - \Ess\right), \\
    \theta E_6^{(3)}=\frac{1}{2}\left(2\Ett \Ess-(\Eff)^{2}-\frac{(\Ess)^2}{\Eff}\right).
    \end{cases}
\end{equation}
In this section, for $N=2,3$, using the Ramanujan identities \eqref{Ramanujan2} and \eqref{Ramanujan3} for \(\Gamma_0^+(N)\), we investigate common zeros of quasi-modular forms for \(\Gamma_0^+(N)\). We first establish Proposition~\ref{gcd} below which let us reduce the question on the common zeros of quasi-modular forms whose Fourier coefficients are algebraic numbers to the question about the greatest common divisor of certain polynomials in $\overline{\Q}[x,y,z]$, based on the idea of Gun and Oesterl\'{e} in \cite{SJ20}. Using Proposition \ref{gcd}, we investigate  common zeros of various quasi-modular forms, as an extension of Meher's result in \cite{MEH}.

\begin{lem}\label{principal}
Let \(N=2\) or 3. For any \(\alpha \in \mathbb H\), the ideal 
\[P_{\alpha}:=\{F \in \overline{\mathbb Q}[x,y,z] : F \left(\Etn(\alpha), \Efn(\alpha),\Esn(\alpha)\right)=0\}
\]
is a principal ideal of $\overline{\Q}[x,y,z]$.
\end{lem}

In order to prove Lemma \ref{principal} we need the following recent result on the algebraic independence for the Eisenstein series, which is a generalization of Nesterenko's theorem \cite{1} for the Fricke groups.

\begin{theorem}\label{IL}{\cite[Theorem 4.6]{IL}} Let $N=1,2,3$. Then,
for any \(\alpha \in \mathbb H\), at least three of the numbers \(e^{2\pi i \alpha}, \Etn(\alpha), \Efn(\alpha), \Esn(\alpha)\) are algebraically independent over \(\mathbb Q\), hence over \(\overline{\mathbb Q}\).
\end{theorem}

\begin{proof}[Proof of Lemma \ref{principal}]
The idea of the proof follows the proof of \cite[Proposition 2]{SJ20}, but we give its self-contained proof for readers. Note that a map 
\[
\overline{\mathbb Q}[x,y,z] \longrightarrow \overline{\mathbb Q}\left[\Etn(\alpha),\Efn(\alpha),\Esn(\alpha)\right]
\]
which sends \(x,y,z\) to \(\Etn(\alpha),\Efn(\alpha),\Esn(\alpha)\), respectively, induces an isomorphism
\[
\overline{\mathbb Q}[x,y,z]/P_{\alpha} \cong \overline{\mathbb Q}\left[\Etn(\alpha),\Efn(\alpha),\Esn(\alpha)\right].
\]
By Theorem \ref{IL}, we know that \(\mathrm{trdeg}_{\overline{\mathbb Q}}\overline{\mathbb Q}\left[\Etn(\alpha),\Efn(\alpha),\Esn(\alpha)\right] \geq 2\), so
\begin{align*}
    \mathrm{ht}(P_{\alpha}) ~&\leq ~\dim \overline{\mathbb Q}[x,y,z] - \dim \overline{\mathbb Q}[x,y,z]/P_{\alpha}\\ &=~3-\mathrm{trdeg}_{\overline{\mathbb Q}}\overline{\mathbb Q}\left[\Etn(\alpha),\Efn(\alpha),\Esn(\alpha)\right]~\leq ~1.
\end{align*}
If \(\mathrm{ht}(P_{\alpha})=0\), then \(P_{\alpha}=\{0\}\), so it is principal. If \(\mathrm{ht}(P_{\alpha})=1\), then since \(\overline{\mathbb Q}[x,y,z]\) is a Noetherian factorization domain, \(P_{\alpha}\) is a principal ideal of $\overline{\mathbb Q}[x,y,z]$.
\end{proof}


Let
\begin{equation}\label{rho}
\rho:\overline{\mathbb Q}\left(\Etn,\Efn,\Esn\right) \to \overline{\mathbb Q} (x,y,z)
\end{equation}
be a natural isomorphism which sends \(\Etn,\Efn,\Esn\) to \(x,y,z\), respectively. Such an isomorphism exists since the functions \(\Etn,\Efn,\Esn\) are algebraically independent over \(\C(e^{2\pi i \tau})\) by \cite[Proposition 6]{ZUD}.

\begin{prop}\label{gcd}
For $N=2,3$, let \(f\) and \(g\) be quasi-modular forms in \(\overline{\mathbb Q}(\Etn,\Efn,\Esn)\) which have a common zero in \(\mathbb H\). Write \(\rho(f)=\frac{f_1}{f_0},\text{ and } \rho(g)=\frac{g_1}{g_0}\) in reduced forms, where \(f_0,f_1,g_0,g_1 \in \overline{\mathbb Q}[x,y,z]\). Then \(f_1\) and \(g_1\) have a non-unit common divisor. 
\end{prop}

\begin{proof} 
Denote the common zero of \(f\) and \(g\) by \(\alpha\). It is obvious that \(f_1,g_1 \in P_{\alpha}\), so they have a non-unit common divisor, namely, a generator of \(P_{\alpha}\).
\end{proof}


\begin{cor}
For $N=2$ or $3$, let $f$ and $g$ be a quasi-modular forms with algebraic Fourier coefficients for $\Gamma_0^+(N).$ Let $h:=\rho^{-1}(\ngcd(\rho(f),\rho(g)))$ where $\ngcd(F,G)$ is the gcd of the numerators of reduced forms of $F$ and $G$ in $\overline{\Q}(x,y,z)$. Then
\begin{align*}
    \Zc(f,g)=\Zc(h),
\end{align*}
where $\Zc(f_1,f_2, \ldots, f_r)$ denotes the common zero set of $f_1,f_2, \ldots, f_r$ in $\Ha$.
\end{cor}

Since \(x\) and \(c_0(y,z)+c_1(y,z)x+\cdots+c_n(y,z)x^n\) for some \(c_i(y,z) \in \overline{\mathbb Q}[y,z]\) have no non-unit common divisor if \(c_0(y,z)\neq 0\), we get the following proposition.

\begin{prop}\label{prop_ratandE2} Let \(N=2\) or $3$.
Let \(f=f_0+f_1\Etn\cdots+f_n\left(\Etn\right)^n\) be a quasi-modular form of depth \(n\) for \(\Gamma_0^+(N)\) such that \(f_0\) is not identically zero and \(f_i \in \overline{\mathbb Q}\left(\Efn,\Esn\right)\) for \(1 \leq i \leq n\). Then, there is no common zero of \(f\) and the Eisenstein series \(\Etn\). In particular, there is no common zero of \(\theta^m \Etn\) and \(\Etn\) for any positive integer \(m\). 
\end{prop}

\begin{proof}
Referring to \eqref{rho} for $\rho$, since \(\rho(\Etn)=x\) and the numerator of a reduced form of \(\rho(f)=\rho(f_0)+\rho(f_1)x+\cdots+\rho(f_n)x^n\) has no non-unit common divisor, Proposition \ref{gcd} implies that \(\Etn\) and \(f\) have no common divisor.
\end{proof}

\begin{prop}\label{prop1.3}
Let \(N=2\) or 3 and \(\alpha \in \mathbb H\) be such that \(j_N(\alpha)\) is an algebraic number. Then \(\theta^nE_2^{(N)}(\alpha)\) is trascendental and hence non-zero for any positive integer \(n\).
\end{prop}

\begin{proof}
Note that \(j_N \in \overline{\mathbb Q}\left(\Efn,\Esn\right)\). Suppose \(j_N(\alpha)\) is an algebraic number, where \(\alpha \in \mathbb H\). This means \(\Efn(\alpha)\) and \(\Esn(\alpha)\) are algebraically dependent over \(\mathbb Q\). If \(e^{2\pi i \alpha}\) and \(\theta^n \Etn(\alpha)\) are algebraically dependent over \(\mathbb Q\), then
\begin{equation*}
    \text{tr}\deg _{\mathbb Q} \mathbb Q\left(e^{2 \pi i \alpha},\Efn(\alpha),\Esn(\alpha),\theta^n \Etn(\alpha)\right) \leq 2.
\end{equation*}
Note that \(\theta^n \Etn =f_0 +f_1\Etn\cdots+f_{n+1}(\Etn)^{n+1}\), where \(f_i\)'s are modular forms of weight $2(n-i+1)$ in \(\overline{\mathbb Q}\left(\Efn,\Esn\right)\). That is, \(f_i\) can be written as
\begin{equation*}
    f_i=\frac{\sum_{u,v} \alpha_{u v}^{(i)} \left(\Efn\right)^{u}\left(\Esn\right)^v}{\sum_{r,s} \beta_{r s}^{(i)} \left(\Efn\right)^r\left(\Esn\right)^s},
\end{equation*}
where \(\alpha^{(i)}_{u v}\) and \(\beta^{(i)}_{r s}\) are algebraic numbers. Thus \(\Etn\) is algebraic over \(\mathbb Q\left(\Efn,\Esn,\theta^n \Etn\right)\) and so
\begin{align*}
    \text{tr}\deg_{\mathbb Q}\mathbb Q \left(e^{2\pi i \alpha},\right.&\left.\Etn(\alpha), \Efn(\alpha),\Esn(\alpha)\right) \\
    & \leq \text{tr}\deg_{\mathbb Q}\mathbb Q\left(e^{2\pi i \alpha},\Efn(\alpha),\Esn(\alpha), \theta^n \Etn(\alpha)\right) \leq 2,
\end{align*}
which contradicts Theorem \ref{IL}. Hence, \(e^{2\pi i \alpha}\) and \(\theta^n \Etn(\alpha)\) are algebraically independent over $\Q$, and in particular, $\theta^n E_2^{(N)}(\alpha)$ is a transcendental number.
\end{proof}

Let 
\begin{equation}\label{eqn: m and d}m = 
    \begin{cases}
        3 & \text{if }N=1, \\
        4 & \text{if }N=2, \\
        6 & \text{if }N=3,
    \end{cases} \quad  \text{ and }\quad 
    d =  
    \begin{cases}
        12 & \text{if }N=1, \\
        \phantom{1}8 & \text{if }N=2, \\
        \phantom{1}6 & \text{if }N=3.
    \end{cases}
\end{equation}

For $N=1,2,3$, let \(f\) be a modular form of weight \(k\) for \(\Gamma_0^+(N)\). Define \[\vartheta f:=\theta f -\frac{k}{d}\Etn f.\]

By the equations (\ref{eq_1}) and \eqref{eq_1(3)} in the proof of Lemma \ref{lem_slash23}, the operator \(\vartheta\) sends \(f\) to a modular form of weight \(k+2\) for \(\Gamma_0^+(N)\). This derivation is a kind of the {\it Serre derivative}. In general, the Serre derivative of a non-cocompact Fuchsian group $\Gamma$ is defiened as follows: Let $\phi$ be a quasi-modular form of weight 2 and depth 1 for $\Gamma$ which is not modular. By multiplying by an appropriate constant, we may assume that $Q_1(\phi)=1$. Then, the Serre derivative of weight $k$ for $\Gamma$ is
\begin{align*}
    \vartheta_{\Gamma}:=\theta -\frac{k}{2\pi i}\phi.
\end{align*}
See \cite[p.48 and p.62]{BRU} for more details. 

Let \(\Delta_N=(\Efn)^3-(\Esn)^2\).


\begin{lem}\label{lem4.2}
For \(N=1\), 2 or 3, let \(f\) be a modular form of weight \(k\) for \(\Gamma_0^+(N)\), which is expressed as a polynomial in \(\Efn\) and \(\Esn\) over \(\mathbb C\). 
\begin{enumerate}[(a)]
    \item If \(N=1\), then \(\vartheta f\) is identically zero if  and only if \(f \in \mathbb C \Delta_1^r\) for some \(r \geq 1\).
    \item If \(N=2\) or 3, then \(\vartheta f\) is not identically zero.
\end{enumerate}
\end{lem}

\begin{proof}
    (a) Suppose \(\vartheta f \equiv 0\), i.e., \(\theta f = \dfrac{k}{12}E_2^{(1)}f\). By the Ramanujan identity \eqref{rama}, we have
    \[
    D'\rho(f)=\dfrac{k}{12}x\rho(f),
    \]
    where 
    \begin{align*}
        D'&=q\dfrac{\partial}{\partial q}+\dfrac{1}{12}(x^2-y)\dfrac{\partial}{\partial x}+\dfrac{1}{3}(xy-z)\dfrac{\partial}{\partial y}+\dfrac{1}{2}(xz-y^2)\dfrac{\partial}{\partial z}.
    \end{align*}
    By the proof of \cite[Lemma 4.1]{1}, \(\rho(f)\) is of the form \(c\Delta^r y^s\) for \(r,s \in \mathbb Z_{\geq 0}\) and a constant $c\in \C$,  where \(\Delta=\rho(\Delta_1)=y^3-z^2\). Moreover, since \(\rho(f) \in \mathbb C[y,z]\), \(\rho(f)\) is of the form \(c\Delta^r\). Conversely, let \(f=c\Delta_1^r\) for \(r \geq 1\). Then it is easy to see that
    \[
    \theta f=\theta(c\Delta_1^r)=rcE_2^{(1)}\Delta_1^r=rE_2^{(1)}f.
    \]
    Since \(\Delta_1\) has weight $12$ for $\SL_2(\Z)$, \(f\) has weight \(k=12r\), so \(r=k/12\). This proves \(\vartheta f \equiv 0\).
    
    (b) Let $m$ and $d$ be integers as in \eqref{eqn: m and d}. Suppose \(\vartheta f \equiv 0\), i.e., \(\theta f=\dfrac{k}{d}\Etn f\). Let
    \begin{align*}
        D_N'&=qy\dfrac{\partial}{\partial q}+\dfrac{m-2}{4m}(x^2-y)\dfrac{\partial}{\partial x}+\dfrac{m-2}{m}(xy-z)\dfrac{\partial}{\partial y}+\dfrac{3(m-2)}{2m}\left(xyz-\dfrac{1}{2}y^3-\dfrac{m-3}{m}z^2\right)\dfrac{\partial}{\partial z}.
    \end{align*}
    By the Ramanujan identities \eqref{Ramanujan2} and \eqref{Ramanujan3}, we have \(\rho(\Efn \theta f)=D_N' \rho(f)\), thus
    \[
    D_N' \rho(f)=\dfrac{k}{d}x y \rho(f).
    \]
    By the same argument as in the proof of \cite[Theorem 3.2]{IL}, \(f\) must be of the form
    \[
    f=c \Delta_N^{\frac{k}{d}} (\Efn)^{-\frac{2k(m-3)}{d(m-2)}}, \quad \text{ for some } c\in \mathbb C.
    \]
    Since $\rho(f) \in \mathbb C[y,z]$, we have that \(\dfrac{k}{d} \in \mathbb Z\). Thus for $N=2$, we have 
    \begin{align*}
        f=c\left(\frac{\Delta_2}{E_4^{(2)}}\right)^{\frac{k}{8}},
    \end{align*}
    which is not holomorphic form. Similarly for $N=3$,
    \begin{align*}
        f=c\left(\frac{\Delta_3}{(E_4^{(3)})^{\frac{3}{2}}}\right)^{\frac{k}{6}},
    \end{align*}
    which is not holomorphic.
    This contradicts the assumption that \(f\) is a modular form of weight~\(k\). Hence \(\vartheta f\) is not identically zero.

\end{proof}

\begin{prop}\label{propC}
For \(N=1\), 2 or 3, let \(f\) be a modular form of weight \(k\) for \(\Gamma_0^+(N)\), which is expressed as a polynomial in \(\Efn \) and \(\Esn\) over \(\mathbb C\). Then,
\begin{enumerate}[(a)]
    \item if \(N=1\), then \(f'\) and \(\Etn=E_2\) have infinitely many common $\SL_2(\Z)$-inequivalent zeros if and only if \(f \in \mathbb C \Delta_1^r\) for some \(r \geq 1\), and
    \item if \(N=2\) or 3, then \(f'\) and \(\Etn\) have finitely many common  \(\Gamma_0^+(N)\)-inequivalent zeros.
\end{enumerate}
\end{prop}

\begin{proof}
    (a) Suppose \(f'\) and \(\Etn\) have infinitely many common $\SL_2(\Z)$-inequivalent zeros. Then \(\vartheta f=\theta f -\dfrac{k}{12}\Etn f\) also has infinitely many $\SL_2(\Z)$-inequivalent zeros. Since any non-zero modular form has only finitely many $\SL_2(\Z)$-inequivalent zeros, \(\vartheta f\) must be 0. By Proposition \ref{lem4.2}, \(f \in \mathbb C \Delta_1^r\). Conversely, if \(f=c\Delta_1^r
    \) for some \(c \in \mathbb C\), then \(\theta f=rcE_2^{(1)}\Delta_1^r\), so every inequivalent zero of \(E_2^{(1)}\) is a zero of \(f'\), and there are infinitely many of them.
    
    (b) Suppose that \(f'\) and \(\Etn\) have infinitely many common \(\Gamma_0^+(N)\)-inequivalent zeros. By the same argument as in the proof of (a), \(\theta f\) must be $0$, which is a contradiction to Lemma~\ref{lem4.2}(b).
\end{proof}

\begin{prop}\label{prop_f'andE2}
For \(N=1, 2\) or $3$, let \(f\) be a modular form of weight \(k\) for \(\Gamma_0^+(N)\), which is expressed as a polynomial in \(\Efn\) and \(\Esn\) over \(\overline{\mathbb Q}\).
\begin{enumerate}[(a)]
    \item If \(N=1\), then \(f'\) and \(\Etn=E_2\) have no common zero if and only if \(f \notin \mathbb C \Delta_1^r\) for any $r \geq 1$.
    \item If \(N=2\) or 3, then \(f'\) and \(\Etn\) have no common zero.
\end{enumerate}
\end{prop}

\begin{proof}
    (a) Suppose \(f \in \mathbb C \Delta_1^r\). Then by Proposition \ref{propC}.(a), \(f'\) and \(E_2^{(1)}\) have infinitely many common zeros. Suppose \(f \notin \mathbb C \Delta_1^r\). Then \(\vartheta f\) is a non-zero modular form, by Lemma~\ref{lem4.2}(a). Since \(f \in \overline{\mathbb Q}\left[E_4^{(1)},E_6^{(1)}\right]\), \(\vartheta f\) must be in \(\overline{\mathbb Q}\left[E_4^{(1)},E_6^{(1)}\right]\). Assume that \(f'\) and \(E_2^{(1)}\) have a common zero, say \(\alpha\). Then \(E_2^{(1)}(\alpha)=0\), so \(E_4^{(1)}(\alpha)\) and \(E_6^{(1)}(\alpha)\) must be algebraically independent by Theorem~\ref{IL}. However, \(\vartheta f(\alpha)=0\) implies that \(E_4^{(1)}(\alpha)\) and \(E_6^{(1)}(\alpha)\) are algebraically dependent. This is a contradiction, hence \(f'\) and \(E_2^{(1)}\) have no common zero.
    
    (b) Suppose \(f'\) and \(\Etn\) have a common zero, say \(\beta\). Then \(\beta\) is a zero of non-zero modular form \(\vartheta f\in \overline{\mathbb Q}\left[\Efn,\Esn\right]\). This yields a contradiction by the same method as in the proof of~(a).
\end{proof}

We mention that in \cite[Theorem 3.5.(iii)]{MEH} the condition $f \notin \C \Delta$ is supposed to be $f \notin \C \Delta^r$ for any $r\geq 1$ as in Proposition~\ref{prop_f'andE2}~(a).

\subsection{Common zeros for $N=1$}

In this subsection, we provide further examples of quasi-modular forms for $\SL_2(\Z)$ with algebraic Fourier coefficients which have no common zeros.

\

Recall that for a non-cocompact Fuchsian group $\Gamma$, we denote the space of quasi-modular forms of weight $k$ and depth $\leq \ell$ for $\Gamma$ by $\widetilde{M}_k^{(\leq \ell)}(\Gamma)$. Also we denote the space of modular forms of weight $k$ for $\Gamma$ by $M_k(\Gamma).$ By abuse of notation, we write $M_k:=M_k(\Gamma)$ and $\widetilde{M}_k^{(\leq \ell)}:=\widetilde{M}_k^{(\leq \ell)}(\Gamma)$ if $\Gamma=\SL_2(\Z)$.

Let $M_{\overline{\Q}}:=\left( \bigoplus_{k \geq 0} M_k \right) \cap \rho^{-1}(\overline{\Q}[y,z])$ be the ring of modular forms for $\SL_2(\Z)$ with algebraic Fourier coefficients, and let $\widetilde{M}_{\overline{\Q}}:= \left( \bigcup_{\ell \geq 0} \bigoplus_{k \geq 0} \widetilde{M}_k^{(\leq \ell)} \right) \cap \rho^{-1}(\overline{\Q}[x,y,z])$ be the ring of quasi-modular forms for $\SL_2(\Z)$ with algebraic Fourier coefficients.

\begin{prop}\label{prop_firstelimination}
Let $f$ and $g$ be quasi-modular forms for $\SL_2(\Z)$, both of algebraic Fourier coefficients. Write $g$ as
\begin{align*}
    g=g_0+g_1 E_2+\cdots +g_{\ell}E_2^{\ell}
\end{align*}
for $g_i \in M_{\overline{\Q}}.$ If there exists a modular form $h \in f\widetilde{M}_{\overline{\Q}} + g\widetilde{M}_{\overline{\Q}}$ such that $h$ is relatively prime to $g_i$ in the ring of modular forms $M_{\overline{\Q}}$ for some $i \in \{0,1,\ldots,\ell\},$ then $f$ and $g$ have no common zero.
\end{prop}

\begin{proof}
Let $\gcd(\rho(f),\rho(g))=d.$ Since $h \in f\widetilde{M}_{\overline{\Q}} + g\widetilde{M}_{\overline{\Q}}$ and is a modualr form, $\rho(h)$ belongs to the first elimination ideal $(\rho(f),\rho(g))_1:=(\rho(f),\rho(g)) \cap \overline{\Q}[y,z]$ of $(\rho(f),\rho(g))$. Thus $\rho(h)$ is divisible by $d$, but here there exists $\rho(g_i)$ for some $i \in \{0,1,\ldots,\ell\}$ which is relatively prime to $\rho(h)$, so prime to $d$. Since $\rho(h) \in \overline{\Q}[y,z]$ implies that $d \in \overline{\Q}[y,z]$ and since $\rho(h) \mid \rho(g)$, we have $d \mid \rho(g_j)$ for any $j \in \{0,1,\ldots,\ell\}.$ In particular $d$ divides both of $\rho(h)$ and $\rho(g_i)$, therefore $d$ must be a constant. 
\end{proof}

\begin{cor}
Let $f$ and $g$ be quasi-modular forms for $\SL_2(\Z)$, both of algebraic Fourier coefficients. Suppose that $g$ has the $f$-expansion, i.e.
\begin{align*}
    g=g_0 + g_1 f + \cdots + g_{\ell} f^{\ell}
\end{align*}
for some $g_0, g_1, \ldots, g_{\ell} \in M_{\overline{\Q}}.$ If $g_0$ is relatively prime to $g_i$ in $M_{\overline{\Q}}$ for some $i \in \{1,2,\ldots,\ell\},$ then $f$ and $g$ have no common zero.
\end{cor}

\begin{proof}
It follows from Proposition \ref{prop_firstelimination} with $h=g_0=g-(g_1+g_2 f+\cdots + g_{\ell}f^{\ell-1})f$.
\end{proof}

If we let $f=E_2$, the above corollary gives Proposition \ref{prop_ratandE2} for $N=1$ (or, \cite[Proposition 3.2]{MEH}).

\begin{cor}
Let $f$ be a modular form of arbitrary weight for $\SL_2(\Z)$, and $g$ be a quasi-modular form for $\SL_2(\Z),$ both of algebraic Fourier coefficients. If there is a quasi-modular form $F \in \widetilde{M}_{\overline{\Q}}$ such that $g$ has the $F$-expansion $g=g_0 + g_1 F + \cdots g_{\ell} F^{\ell}$ for which $g_i$'s are coprime in $M_{\overline{\Q}},$ then $f$ and $g$ have no common zero.
\end{cor}

\begin{proof}
It follows from Proposition \ref{prop_firstelimination} with $h=f$, then since one of $g_i$ must be relatively prime to $h$.
\end{proof}

\begin{example}
There is no common zero of any holomorphic modular form (equivalently, a quasi-modular form of the minimal depth $r=0$) and a quasi-modular form of the maximal depth (that is, a quasi-modular form of weight $k$ and depth $r=\frac{k}{2}$) for $\SL_2(\Z)$.
\end{example}

\begin{example}
Let $k=2,4,6$. For each integer $n\geq 1$ there is no common zero of $\frac{d^n E_k(\tau)}{d\tau^n}$ and any non-zero holomorphic modular form of arbitrary weight for $\SL_2(\Z)$. 
\end{example}

\begin{rmk}
Let $I \subseteq \C[x_1,x_2,x_3]$ be the vanishing ideal of a singleton $\{(1,1,1)\} \subset \Ac^3$, and let $J$ be a weighted homogeneous ideal contained in $I.$ Then the first elimination ideal $J_1:=J \cap \C[x_2,x_3]$ of $J$ is contained in $(x_2^3-x_3^2)\C[x_2,x_3]$. Indeed, if we pick an arbitrary generator $F$ of the ideal $J$, then $f:=\rho^{-1}(F)$ is a quasi-modular form of weight $2\wtdeg(F)$, and is cuspidal. Since $J$ is weighted homogeneous, so is $J_1$, which means that any weighted homogeneous element of $J_1$ is also cuspidal, i.e., it is corresponding to a holomorphic cusp form. Note that any cusp form is written as $\Delta_1 \cdot g$ for some modular form $g$ and the weight $12$ cusp form $\Delta_1:=E_4^3-E_6^2.$ Hence any generator of an ideal $J_1$ is divided by the polynomial $x_2^3-x_3^2.$ 
\end{rmk}

\section{Simplicity of zeros of quasi-modular forms: Proof of Theorem~\ref{thm: simple zero quasimodular}}  \label{sec: simplicity}

In this section, we show the simplicity of zeros, Theorem~\ref{thm: simple zero quasimodular}. \cref{gcd} provides a useful method to prove the non-existence of common zeros of certain quasi-modular forms.  For example, one can easily deduce the following proposition from \cref{gcd}.

\begin{prop}\label{prop: simple zero 1}
For \(N=2,3\), all the zeros of each of  \[\Etn,\theta \Etn,\theta^2\Etn,\theta^3 \Etn\] are simple. Moreover, there are no common zeros of any two of \[\Etn,\theta \Etn, \theta^2 \Etn, \theta^3 \Etn, \theta^4 \Etn.\]
\end{prop}

\begin{proof}
By the Ramanujan identity \eqref{Ramanujan2} for $\Gamma_0^+(2)$, we have
\begin{align*}
        \theta \Et&=\dfrac{1}{8}\left((\Et)^2-\Ef\right), \\
        \theta^2 \Et&=\dfrac{1}{32}\left((\Et)^3-3\Et \Ef+2\Es\right), \\
        \theta^3 \Et&=\dfrac{1}{256}\left(3(\Et)^4-18(\Et)^2 \Ef+24\Et\Es -5(\Ef)^2-4\dfrac{(\Es)^2}{\Ef} \right), \\
        \theta^4 \Et&=\dfrac{1}{512}\left(3(\Et)^5-30(\Et)^3\Ef+60(\Et)^2\Es-25\Et(\Ef)^2\phantom{\dfrac{\Et(\Es)^2}{\Ef}}\right. \\ &\left.\quad +12\Ef\Es-20\dfrac{\Et(\Es)^2}{\Ef} \right),
\end{align*}
so the numerators of reduced forms of \(\rho(\theta^i \Et)\) for \(i=0,1,2,3,4\) are
\begin{align*}
    &x, \\
    &x^2-y, \\
    &x^3-3xy+2z, \\
    &3x^4y-18x^2y^2+24xyz-5y^3-4z^2, \\
    &3x^5y-30x^3y^2+60x^2yz-25xy^2+12y^2z-20xz^2,
\end{align*}
respectively. The proof for the case when $N=2$ follows from the fact that any two numerators listed above have no non-unit common divisors.
Similarly, for \(N=3\), by the Ramanujan identity \eqref{Ramanujan3} for $\Gamma_0^+(3)$, we have
\begin{align*}
    \theta \Ett&=\dfrac{1}{6}\left((\Ett)^2-\Eff \right), \\
    \theta^2 \Ett&=\dfrac{1}{18}\left((\Ett)^3-3\Ett \Eff+2\Ess \right), \\
    \theta^3 \Ett&=\dfrac{1}{36}\left((\Ett)^4-6(\Ett)^2\Eff-(\Eff)^2+8\Ett \Ess-2\dfrac{(\Ess)^2}{\Eff} \right), \\
    \theta^4 \Ett&=\dfrac{1}{54}\left((\Ett)^5-10(\Ett)^3\Eff-5\Ett(\Eff)^2+20(\Ett)^2\Ess\right.\\
    &\left.\quad +3\Eff \Ess-10\dfrac{\Ett(\Ess)^2}{\Eff}+\dfrac{(\Ess)^3}{(\Eff)^2} \right),
\end{align*}
so the same argument completes the proof for \(N=3\).
\end{proof}

This method can be adopted to show the the simplicity of zeros of all derivatives of quasi-modular forms $E_2^{(N)}$, $E_4^{(N)}$, $E_6^{(N)}$ by analyzing the rational functions $\rho\left(E_2^{(N)}\right)$, $\rho\left(E_4^{(N)}\right)$, $\rho\left(E_4^{(N)}\right)$ and the derivations $D^{(N)}$ in $\QQ(x, y, z)$ corresponding to $\theta$ for $N=1,2,3$.

In \cite[Theorem 7]{SJ20}, the authors considered $N=1$ cases and proved that all the zeros of $\frac{d^r E_k^{(1)}(\tau)}{d\tau^r}$ are simple for all integers $r\ge1$ and even integers $k\ge2$. Their method, which relies on the properties of certain polynomial rings, has some obstruction to be generalized to the cases when $N\geq 2$, since the Ramanujan identities \eqref{Ramanujan2} and \eqref{Ramanujan3} for $N=2,3$  are not of polynomial forms in terms of $E_k^{(N)}$'s anymore. So we carry out some careful analysis on this issue to generalize the simplicity result to the cases when $N=2, 3$. For the readers' convenience, we convey its complete proof including the case $N=1$.

With the natural isomporphism $\rho\colon  \mathbb{Q}\left(E_2^{(N)},E_4^{(N)},E_4^{(N)}\right)\to \mathbb Q(x,y,z)$ such that $\rho\left(E_2^{(N)}\right)=x$, $\rho\left(E_4^{(N)}\right)=y$, $\rho\left(E_6^{(N)}\right)=z$  given in \eqref{rho}, the derivation $\theta$ on the space of quasi-modular forms can be represented as the derivation $D^{(N)}$ on $\mathbb{Q}(x,y,z)$ as follows:
\begin{align} \label{eqn: DN defn}
    D^{(N)} f &:=
    \begin{cases}
    \frac{x^2-y}{12} \frac{\partial}{\partial x} f + \frac{xy-z}{3} \frac{\partial}{\partial y}f + \frac{zx - y^2}{2} \frac{\partial}{\partial z} f , & \text{if }N=1, \\ 
    \frac{x^2- y}{8} f_x + \frac{xy-z}{2} f_y + \frac{3xz - 2y^2 - z^2/y}{4} f_z , & \text{if }N=2, \\ 
    \frac{x^2- y}{6} f_x + \frac{2(xy-z)}{3} f_y + \frac{2xz - y^2 - z^2/y}{2} f_z ,  & \text{if }N=3, 
    \end{cases} \\ 
    & = p^{(N)} f_x + q^{(N)} f_y + r^{(N)} f_z, \notag
\end{align}
where 
\begin{align*}
    &p^{(1)}:=\frac{x^2-y}{12},&& q^{(1)} := \frac{xy-z}{3},&& r^{(1)} := \frac{zx - y^2}{2},\\
    &p^{(2)}:=\frac{x^2- y}{8},&&  q^{(2)} := \frac{xy-z}{2},&&  r^{(2)} := \frac{3xz - 2y^2 - z^2/y}{4}, \\
    &p^{(3)}:= \frac{x^2- y}{6},&& q^{(3)} := \frac{2(xy-z)}{3},&&  r^{(3)} := \frac{2xz - y^2 - z^2/y}{2},
\end{align*}
referring to~\eqref{rama}, \eqref{Ramanujan2}, and \eqref{Ramanujan3}. Recalling $m$ and $d$ in \eqref{eqn: m and d} depending on $N$, we note that $d=\frac{4m}{m-2}$ for each $N=1,2,3$. Then the above equations can be summarized as 
\begin{align*} D^{(N)}f =&  \frac{1}{d}\left( (x^2-y) f_x + 4(xy-z) f_y +\left( 6xz-\frac{2m}{m-2}y^2 -\frac{4(m-3)}{m-2}z^2/y\right)  f_z\right) .
\end{align*}

Note that for $N=2,3$, $yD^{(N)}$ are operators in $\QQ[x,y,z]$, but $D^{(N)}$ are not. Before  further discussion on common zeros, we need to carry out some careful analyses on the degrees of the denominators of $(D^{(2)})^n$ and $(D^{(3)})^n$ at $x$, $y$, and $z$, for $n\ge1$.

\subsection{
\texorpdfstring{Degree analysis on $\left(D^{(N)}\right)^n$}{Degree analysis on (D⁽ᴺ⁾)ⁿ}
} \label{section: degree analysis}
This section is dedicated to some necessary analyses of the degrees of $\left(D^{(N)}\right)^n x$, $\left(D^{(N)}\right)^n y$, and $\left(D^{(N)}\right)^n z$ for $N=2, 3$, and especially the degrees of their denominators. 

We consider the case when $N=2$ in detail here. The results for the case when $N=3$ will be stated in this section, and its proof will be given in Appendix~\ref{sec: appendix}.

Let $\mathcal{D}^{(2)} = dD^{(2)} = 8 D^{(2)}$ so that they have the integral coefficients, i.e.,
\[ \mathcal{D}^{(2)} = (x^2 - y) \frac{\partial}{\partial x} + (4xy - 4z) \frac{\partial}{\partial y} + (6xz - 4y^2 - 2 y^{-1}z^2 ) \frac{\partial}{\partial z}.
\]

Let $\deg_x, \deg_y, \deg_z$ be the degrees of monomials in $\QQ[x,y, z, 1/y]$ with respect to $x$, $y$ and $z$, respectively. (Note that $\deg_y$ can be negative.) 
Also, let \[\wtdeg = \deg_x + 2 \deg_y + 3\deg_z\] be the weighted degree. Note that all monomials in $\left( D^{(N)}\right)^n x$, $\left( D^{(N)}\right)^n y$, $\left( D^{(N)}\right)^n z$ have the same $\wtdeg$, so $\wtdeg$ can be extended to them, for $n\ge0$. Also we note that $\wtdeg \left(\mathcal{D}^{(2)} f\right) = \wtdeg(f) +1$ for a weighted homogeneous $f\in \QQ[x, y, z, 1/y]$.

Since 
\begin{alignat*}{2} \mathcal{D}^{(2)} \left( x^a y^b z^c\right)  = (a+4b + 6c) x^{a+1}y^b z^c - (4b +2c) x^a y^{b-1} z^{c+1}   -4c x^a y^{b+2} z^{c-1} -a x^{a-1}y^{b+1} z^c, \end{alignat*}
when we represent $\textsf{(some coefficient)}\cdot x^a y^b z^c$ as a point $(a, b,c)^\tp \in \mathbb{Z}^3$ (here, ${}^\tp$ stands for the transpose) $\mathcal{D}^{(2)}$ transfers $(a, b,c)^\tp$ to (at most) four points as 
\[ \begin{pmatrix} a\\ b\\c\end{pmatrix} \stackrel{ \mathcal{D}^{(2)}}{\longmapsto} \begin{pmatrix}a\\ b\\ c\end{pmatrix} + \begin{pmatrix} 0 \\ 1/2 \\ 0\end{pmatrix} + \left\{ 
\pm \begin{pmatrix} \phantom{-}0 \\ - 3/2 \\ \phantom{-}1 \end{pmatrix}, \pm \begin{pmatrix} \phantom{-}1\\ -1/2 \\ \phantom{-}0 \end{pmatrix}\right\}.\]
We introduce the matrix,
\begin{equation}\label{eq:T}
T = \frac{1}{14}\begin{pmatrix} 13 & -2 & -3 \\ 3 & 6 & -5 \\  1 & 2 & 3\end{pmatrix} \text{ with its inverse } T^{-1} = \begin{pmatrix}1 & 0 & 1 \\-\frac{1}{2} & \frac{3}{2} & 2 \\0 & -1 & 3 \end{pmatrix}, 
\end{equation}
which is determined to have the following relation, and transforms the above four candidate points into the unit vectors on a plane:
\[ T\begin{pmatrix} a\\ b\\c\end{pmatrix} \stackrel{T \mathcal{D}^{(2)}T^{-1}}{\longmapsto} T\begin{pmatrix}a\\ b\\ c\end{pmatrix} + T\begin{pmatrix} 0 \\ 1/2 \\ 0\end{pmatrix} + \left\{
\pm \begin{pmatrix} 0 \\ 1 \\ 0 \end{pmatrix}, \pm \begin{pmatrix}1 \\ 0 \\ 0 \end{pmatrix}\right\},\]
and by this way we can represent the monomial $x^a y^b z^c$ in $\left(\mathcal{D}^{(2)}\right)^r(x)$ as $(\lambda, \nu) \in \mathbb{Z}^2$  with relation 
\begin{equation} \label{eqn: monomial to plane}
(a, b, c)^\tp = (1,r/2,0)^\tp + T^{-1} (\lambda, \nu, 0)^\tp, 
\end{equation}
i.e.,  $(a, b, c) = \left(\lambda+1, (-\lambda +3\nu +r)/2, -\nu \right)$. 
After applying $\left(\mathcal{D}^{(2)}\right)^r$, the monomials of $\left(\mathcal{D}^{(2)}\right)^r x$ are represented as elements in 
\[\left\{(1,r/2,0)^\tp +T^{-1} (\lambda, \nu, 0)^\tp : |\lambda |+|\nu| \le r, \quad \lambda+ \nu \equiv r \pmod 2\right\}.\]

This set is not necessarily the whole set of all monomials in $\left( \mathcal{D}^{(2)}\right)^r x$; for example, there are some obvious restrictions like the non-negativeness of $\deg_x$ and $\deg_z$, which means $\lambda \ge-1$  and $\nu \le 0$. We can prove  another condition and find the formula for the minimal $\deg_y$ of  monomials of $\left( \mathcal{D}^{(2)}\right)^r x$.

Let $x^a y^b z^c$ be a monomial appearing in $\left( \mathcal{D}^{(2)}\right)^r x$. If $\mathcal{D}^{(2)}(x^a y^b z^c)$ contains a monomial with $\deg_y$ less than $b$, then the monomial with decreased $\deg_y$ is  $-(4b +2c) x^a y^{b-1} z^{c+1}$. This term is non-zero only when $c \ne -2b$. When $c=-2b$, since $a+2b+3c = \wtdeg\left(\left( \mathcal{D}^{(2)}\right)^r x\right) = r+1$, we have $a = r+1 - 2b - 3c = r+1-4c$. 
Referring to \eqref{eqn: monomial to plane}, the corresponding condition for $(\lambda, \nu)$ is that $\nu  = \frac{1}{2}\lambda - \frac{r}{2}$, i.e., only the monomials (represented as a lattice point $(\lambda, \nu)$) off this line can provide the monomials whose  $\deg_y$ drops by $1$ after applying $\mathcal{D}^{(2)}$.

\begin{figure}[!htpb]
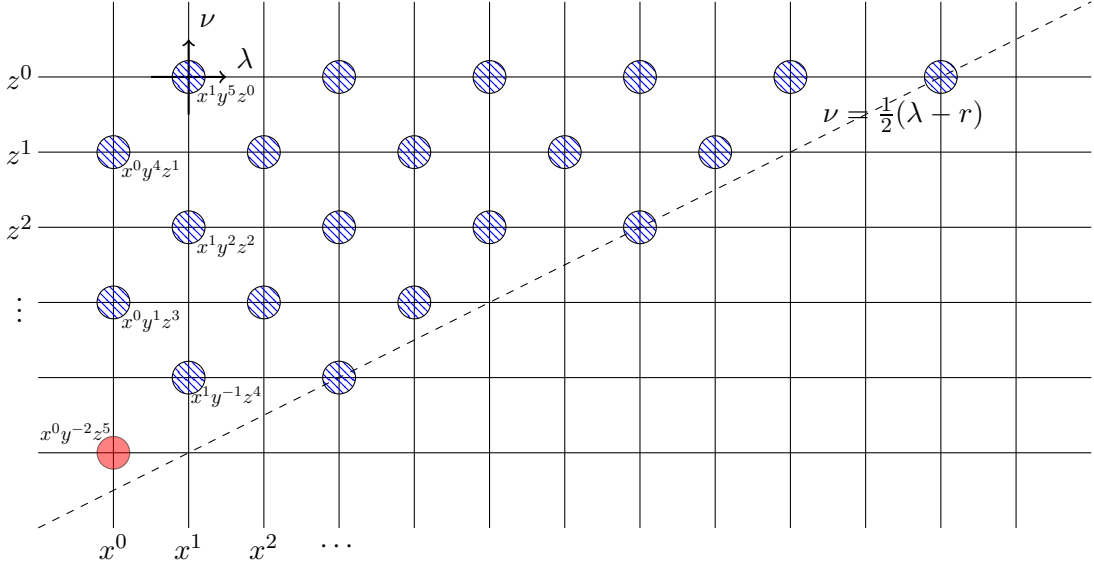

    \centering
    \ctikzfig{fig1}
    \caption{Monomials in $\left( \mathcal{D}^{(2)}\right)^{10}(x)$, represented as lattice points $(\lambda, \nu)$. Note that two bold axes intersects at the origin $(0,0)$, and the red-filled point at $(-1, -5)$ represents $y^{-2}z^5$.}  \label{fig: N=2} 
\end{figure}

Figure~\ref{fig: N=2} shows the situation for $\left( \mathcal{D}^{(2)}\right)^{10} x$. Each points $(\lambda, \nu)$ on the figure represents each monomial $x^{a} y^b z^c$ appearing in $\left(\mathcal{D}^{(2)}\right)^{10}x$ by the relation given in \eqref{eqn: monomial to plane}. The red-filled  point at $(\lambda, \nu) = (-1, -5)$ (which represents $y^{-2}z^{5}$-term) is off the dashed line $\nu  = \frac{1}{2}\lambda  - \frac{r}{2}$. So, 
$\mathcal{D}^{(2)}(y^{-2}z^{5})$ provides a function which contains a monomial whose $\deg_y$ is $ -3$; one can check that $ \mathcal{D}^{(2)}(y^{-2}z^{5}) = 22 x z^5 y^{-2}-2 z^6y^{-3}-20 z^4$. 

From the discussion so far, we conclude the following lemma. 
\begin{lem}\label{lem: mn condition N=2} Let $x^a y^b z^c$ be some monomial in $\left(D^{(2)}\right)^r x$ (up to nonzero coefficients over $\QQ$).
    \begin{enumerate}[(a)]
        \item Its corresponding lattice point $(\lambda, \nu)$ via \eqref{eqn: monomial to plane} satisfies that $\nu \ge \frac{1}{2}(\lambda -r)$. 
        \item $2b+c \ge 0$ and \label{lem: sharp ineq N=2} $b \ge -\frac{r+1}{4}$.
    \end{enumerate}
\end{lem}
\begin{proof}
    \begin{enumerate}[(a)]
        \item This holds for the initial case $r=0$. As $r$ increases, all points keep lying above or on the line
        $\nu  = \frac{1}{2}(\lambda -r)$, since the line itself is shifted by $-1/2$ in $y$ axis direction on each step, and the point on the line does not provide a monomial with decreased $\deg_y$, i.e., decreased $\nu $ coordinate. 
        \item Referring to \eqref{eqn: monomial to plane}, since $a+2b+3c = r+1$, the inequality $\nu \ge \frac{\lambda }{2}- \frac{r}{2}$ is equivalent to $2b+c\ge0$, which implies the second inequality since $a\ge0$ and $a+2b+3c=r+1$.
    \end{enumerate}
\end{proof}

In the remaining section, we prove that the second inequality of Lemma~\ref{lem: mn condition N=2}~\ref{lem: sharp ineq N=2} is sharp, in the sense that when $r\equiv 3 \pmod 4$, there exists a monomial with $\deg_y = -\frac{r+1}{4}$ in $\left(\mathcal{D}^{(2)}\right)^r(x)$.

To prove this, we need to show that the coefficient of the $y^{-k}z^{2k+1}$-term in $\left( \mathcal{D}^{(2)}\right)^{4k+2}$  is non-zero for all integers $k\ge0$. It can be shown by proving that their coefficients are all $1$ modulo $5$. 

Let  $a_{i,j,k}^{(r)}$ denote the coefficient of the $x^i y^j z^k$-term in $\left( \mathcal{D}^{(2)}\right)^{r}x$. 
By keeping track of the recurrences for the coefficients, we reduce the degrees as follows:
\begin{alignat*}{2}
    a_{0, -k, 2k+1}^{(4k+2)} & = (-4(-k+1) - 2(2k))a_{0, -k+1, 2k}^{(4k+1)}\\ 
    & = -4 a_{0, -k+1, 2k}^{(4k+1)} \equiv a_{0, -k, 2k}^{(4k+1)} \pmod 5\\ 
    & = (-4(-k+2) - 2(2k-1)) a_{0, -k+2, 2k-1}^{(4k)} -a_{1,-k, 2k}^{(4k)} \\ 
    & = -6a_{0, -k+2, 2k-1}^{(4k)} - a_{1, -k, 2k}^{(4k)} \equiv -(a_{0, -k+2, 2k-1}^{(4k)} + a_{1, -k, 2k}^{(4k)}) \pmod 5.
\end{alignat*}
Since \begin{alignat*}{2}
\begin{cases}
    a_{0, -k+2, 2k-1}^{(4k)} & = -8 a_{0, -k+3, 2k-2}^{(4k-1)} - a_{1, -k+1, 2k-1}^{(4k-1)} - 8r a_{0, -k, 2k}^{(4k-1)},\\ 
    a_{1, -k, 2k}^{(4k)} &  =  - 2a_{1, -k+1, 2k-1}^{(4k-1)} + 8r a_{0, -k, 2k}^{(4k-1)},
\end{cases}
\end{alignat*}
we have \[a_{0, -k, 2k+1}^{(4k+2)}   \equiv  3\left(a_{1, -k+1, 2k-1}^{(4k-1)}   + a_{0, -k+3, 2k-2}^{(4k-1)}\right)\pmod 5.\] 
Similarly, since 
\begin{alignat*}{2}
\begin{cases}
    a_{0, -k+3, 2k-2}^{(4k-1)} & = (-8k+4) a_{0, -k+1, 2k-1}^{(4k-2)}-10 a_{0, -k+4, 2k-3}^{(4k-2)} - a_{1, -k+2, 2k-2}^{(4k-2)},\\ 
    a_{1, -k+1, 2k-1}^{(4k-1)} & = -2 a_{1, -k+2, 2k-2}^{(4k-2)} + (8k-2) a_{0, -k+1, 2k-1}^{(4k-2)},
\end{cases}
\end{alignat*}
we have \[a_{0,-k,2k+1}^{(4k+2)} \equiv a_{0,-k+1,2k-1}^{(4k-2)}\pmod{5}.\]
Since $a_{0,0,1}^{(2)} \not\equiv 0\pmod 5$, the coefficient $a_{0,-k, 2k+1}^{(4k+2)}$ never vanishes. 

We conclude that $y^{\ell_r} \cdot \left(D^{(2)}\right)^{r} x \in \QQ[x,y,z]\setminus y\QQ[x,y,z]$, with $\ell_r =\lfloor(r+1)/4\rfloor$.

The analyses for $\left(D^{(2)}\right)^r y$ and $\left( D^{(2)}\right)^r z$ are exactly the same; in conclusion, we have the following lemma. 
\begin{lem} \label{lem: y powers N=2}
    For each integer $r\ge0$, if we let $\ell_r =\lfloor(r+1)/4\rfloor$ then
    \[y^{\ell_r}\cdot \left(D^{(2)}\right)^{r}x,\quad y^{\ell_{r+1}}\cdot \left(D^{(2)}\right)^{r}y,\quad y^{\ell_{r+2}}\cdot \left(D^{(2)}\right)^{r}z \in \mathbb{Q}[x,y,z]\setminus y\mathbb{Q}[x,y,z].\] 
\end{lem}

We get similar results for the case when $N=3$ as follows and we postpone their proofs given in~\cref{sec: appendix} so that the readers can see the proof of Theorem~\ref{thm: simple zero quasimodular} just below right away.

\begin{lem}\label{lem: mn condition N=3} Let $x^a y^b z^c$ be some monomial in $\left(D^{(3)}\right)^r x$ (up to some nonzero $\QQ$ coefficient).
    \begin{enumerate}[(a)]
        \item Its corresponding lattice point $(\lambda, \nu)$ via \eqref{eqn: monomial to plane} satisfies that $\nu \ge \frac{2}{3}(\lambda -r)$. 
        \item \label{lem: sharp ineq N=3} $4b+3c \ge 0$ and $b \ge -\frac{r+1}{2}$. 
    \end{enumerate}
\end{lem}
\begin{lem} \label{lem: y powers N=3}For $r\ge0$, let $\ell_r =  \begin{cases}
    \lfloor (r+1)/2 \rfloor -1& \text{if $\ell \equiv 1,2, 3 \pmod 6$}, \\ 
    \lfloor (r+1)/2 \rfloor& \text{otherwise.}
    \end{cases}\quad$ Then,
    \[y^{\ell_r}\cdot \left(D^{(3)}\right)^{r}x,
    \quad y^{\ell_{r+1}}\cdot \left(D^{(3)}\right)^{r}y,
    \quad y^{\ell_{r+2}}\cdot \left(D^{(3)}\right)^{r}z \in \mathbb{Q}[x,y,z]\setminus y\mathbb{Q}[x,y,z].\] 
\end{lem}
\subsection{Proof of the Theorem~\ref{thm: simple zero quasimodular}} 
With the results in Section~\ref{section: degree analysis}, we prove Theorem~\ref{thm: simple zero quasimodular}. 

Recall \eqref{eqn: DN defn} for the definition of $D^{(N)}$. We define the  auxiliary differential operators $D^{(N)}_t$ for $t=x,y,z$, as 
\[D^{(N)}_t:= p^{(N)}_t f_x + q^{(N)}_t f_y + r^{(N)}_t f_z,\]
for each $N=1, 2,3$. Especially $D^{(N)}_x$ plays a special role among others as we see later in this section, thus we let $\widetilde{D^{(N)}}:= D^{(N)}_x$. 

We note that for $N=1,2,3$,
\[ \widetilde{D^{(N)}}f = \left( 2xf_x +4y f_y + 6z f_z\right)/d, \] 
and especially, 
\begin{equation} \label{eqn: DNxyz}
\widetilde{D^{(N)}}x = (2/d)x,\quad  \widetilde{D^{(N)}}y = (4/d)y,\quad  \widetilde{D^{(N)}}z=(6/d)z.
\end{equation}

\ 

We generalize \cite[Lemma 16, Lemma 17]{SJ20} for the case $N= 1$ and get the following two lemmas for $N=2, 3$.
\begin{lem} \label{lem: interplaying N} We have the following relations of $D^{(N)}$, $\widetilde{D^{(N)}}$ and $\px$ for each $N=1,2,3$;
    \begin{enumerate}[(a)]
        \item $\px D^{(N)} f = \widetilde{D^{(N)}} f + D^{(N)}(f_x)$ for $f\in \QQ[x,y,z, 1/y]$. 
        \item  $\widetilde{D^{(N)}} \left(D^{(N)}\right)^n =  \left(D^{(N)}\right)^n \widetilde{D^{(N)}} + \frac{2n}{d}\left(D^{(N)}\right)^n$ for all $n\ge1$.
    \end{enumerate}
\end{lem}
\begin{proof}
    (a) follows from the direct calculation; for $t = x, y, z$,
        \begin{align*}
            \frac{\partial }{\partial t} D^{(N)}f &= p^{(N)}_tf_x + q^{(N)}_t f_y + r^{(N)}_t f_z + p^{(N)} f_{xt} + q^{(N)} f_{yt} + r^{(N)} f_{zt}  = D^{(N)}_t  f + D^{(N)}(f_t).
        \end{align*}

        In order to prove (b), we can verify the following:
        \begin{align*}&A  := \begin{pmatrix} 
            p^{(N)}_x & p^{(N)}_y & p^{(N)}_z \\ 
            q^{(N)}_x & q^{(N)}_y & q^{(N)}_z \\ 
            r^{(N)}_x & r^{(N)}_y & r^{(N)}_z \end{pmatrix} = 
            \frac{1}{d}\begin{pmatrix}
            2x & -1 & 0 \\ 
            4y & 4x & -4 \\ 
            6z & \frac{4 (m-3) z^2}{ (m-2) y^2}- dy  &6x-\frac{8 (m-3) z}{ (m-2) y}\end{pmatrix}, \\ 
            &A\begin{pmatrix}p^{(N)}_x \\ q^{(N)}_x \\ r^{(N)}_x
            \end{pmatrix} = 
            \begin{pmatrix}
                \widetilde{D}p^{(N)} \\
                \widetilde{D}q^{(N)}\\
                \widetilde{D}r^{(N)}
            \end{pmatrix}
            =
            \frac{2}{d}\begin{pmatrix}
                2p\\
                3q \\
                4r
            \end{pmatrix},
        \end{align*}
       and
        \[ B: = \begin{pmatrix} p_{xx} & p_{xy} & p_{xz} \\ 
            q_{xx} & q_{xy} & q_{xz} \\
            r_{xx} & r_{xy} & r_{xz} \end{pmatrix} =
            \frac{2}{d}\begin{pmatrix}
            1& 0 & 0 \\
            0 & 2 & 0 \\
            0 & 0 & 3 \end{pmatrix}, \quad
            B \begin{pmatrix} p \\ q \\ r \end{pmatrix} =
             \begin{pmatrix}
                D p_x \\ 
                D q_x \\
                D r_x \end{pmatrix}
                =
            \frac{2}{d}\begin{pmatrix}
                p \\ 
                2q \\
                3r \end{pmatrix}.
        \] 
        With these equations, the direct calculation shows that $\widetilde{D^{(N)}}D^{(N)}f - D^{(N)}\widetilde{D^{(N)}} f = \frac{2}{d}D^{(N)}f. $ With this and (a), we can show (b) inductively. 
    \end{proof}
    \begin{rmk} The property (b) of Lemma~\ref{lem: interplaying N} is not naturally arising from an arbitrary choice of $p$, $q$ and $r$. The boilage from  tedious calculations of $\widetilde{D^{(N)}}D^{(N)} - D^{(N)} \widetilde{D^{(N)}}$ consists of $D^{(N)}_x(p^{(N)}) - D^{(N)}(p^{(N)}_x)$, $D^{(N)}_x(q^{(N)}) - D^{(N)}(q^{(N)}_x)$ and $D^{(N)}_x(r^{(N)}) - D^{(N)}(r^{(N)}_x)$ terms, which differ from $p,q$ and $r$ only by a scalar  multiple (specifically, $2/d$), respectively. 
    \end{rmk}

\begin{lem} Let $N=1,2,3$. For each integer $n\ge1$, we have 
\label{lem: integral N}
    \begin{align}
        \px \left( (D^{(N)})^n x \right) &= \frac{n(n+1)}{d} (D^{(N)})^{n-1}x, \label{eqn: interplay x power N}\\ 
        \px \left( (D^{(N)})^n y \right) &= \frac{n(n+3)}{d} (D^{(N)})^{n-1}y, \label{eqn: interplay y power N}\\ 
        \px \left( (D^{(N)})^n z \right) &= \frac{n(n+5)}{d} (D^{(N)})^{n-1}z, \label{eqn: interplay z power N}
    \end{align}
    i.e., $\frac{\partial}{\partial x} (D^{(N)})^n x$ is a scalar multiple of $(D^{(N)})^{n-1}x$, and so are those at $y$ and $z$.
\end{lem}
\begin{proof}
    We prove the following by induction on $n$:
    \[\px ((D^{(N)})^n f) = n (D^{(N)})^{n-1} \widetilde{(D^{(N)})} f + (D^{(N)})^n f_x + \frac{n(n-1)}{d} (D^{(N)})^{n-1} f.\]
    
    For $n=1$, we have that $\frac{\partial}{\partial x} D^{(N)}f  =\widetilde{D^{(N)}}f + D^{(N)}(f_x)$ by Lemma~\ref{lem: interplaying N}. To proceed by induction, suppose it holds for $n-1$. Then by Lemma~\ref{lem: interplaying N}, we have that
\begin{align*}
 \frac{\partial}{\partial x} ((D^{(N)})^n f)& =  \frac{\partial}{\partial x} D^{(N)} ({(D^{(N)})}^{n-1} f)\\
&=\widetilde{D^{(N)}}({(D^{(N)})}^{n-1}f)+D^{(N)}\left(\frac{\partial}{\partial x} ({D^{(N)}})^{n-1} f\right)\\
&=({D^{(N)}})^{n-1} \widetilde{D^{(N)}}f + \frac{2(n-1)}{d}{(D^{(N)})}^{n-1}f\\
&\phantom{=}+(n-1)(D^{(N)})^{n-1} \widetilde{D^{(N)}} f + (D^{(N)})^n f_x+\frac{(n-1)(n-2)}{d}(D^{(N)})^{n-1}f\\ 
    & = n (D^{(N)})^{n-1} \widetilde{D^{(N)}} f + (D^{(N)})^n f_x + \frac{n(n-1)}{d} (D^{(N)})^{n-1} f.
\end{align*}

Substituting $x$, $y$, $z$ for $f$ in turn yields the desired result, referring to \eqref{eqn: DNxyz}.
\end{proof}

\begin{lem} \label{lem: nondividing N} Let $N=1,2,3$. Let $f  \in \QQ[x,y,z]$ be a prime factor of the numerator of a reduced form of one of $(D^{(N)})^n x$, $(D^{(N)})^n y$, and $(D^{(N)})^n z$, for some integer $n\ge1$. Then $f$ divides neither $y$ nor the numerator of a reduced form of $D^{(N)} f$, and $f_x \ne0$. 
\end{lem}
\begin{proof}
    Note that the only principal ideals $I$ with property $yD^{(N)}(I) \subseteq I$ are $I=(y)$ or $I=(y^3-z^2)$. (See \cite[Theorem 3.2]{IL}.)

    
    
    Let $n\ge1$ be a given integer. Let $F/y^\ell$ be a reduced form of $(D^{(N)})^n x$ in $\Q(x,y,z)$ for some integer $\ell\ge0$. We claim that $F\not \in y \QQ[x,y,z]$. When $N=1$, since the coefficient of $x^{n+1}$-term in $D^{(n)}x$ is not zero, $F \not \in y\QQ[x,y,z]$. For $N=2,3$, we have shown that $\ell \ge1$, i.e., $y \nmid F$, in Lemma~\ref{lem: y powers N=2} and Lemma~\ref{lem: y powers N=3} when $n\ge3$. Since $(D^{(N)})x$ and $(D^{(N)})^2 x$ don't have $y$ as a factor of each, $F \not\in y\QQ[x,y,z]$ for $n\ge1$.

    Let $f$ be a prime factor of $F$. If $f$ divides the numerator of a reduced form of $D^{(N)}f$, then $f \mid y D^{(N)}f$ (in $\QQ[x,y,z]$), i.e., $f \in (y)$ or $y \in (y^3 - z^2)$. Since $f\not\in (y)$, $f = k(y^3 -z^2)$ for some $k \in \QQ$. Thus, we have $(y^3-z^2)\mid F$.         

    Note that there is a non-vanishing $x^{n+1}$-term in $(D^{(N)})^{n}x$. Since $y^3-z^2$ divides $F$, $(D^{(N)})^nx$ has a non-zero $x^{n+1}y^{-3}z^2$-term   with the same (non-zero) coefficient. 
    This is clearly  impossible when $N=1$, and also impossible when $N=2, 3$, referring to Lemma~\ref{lem: mn condition N=2}~(b) and Lemma~\ref{lem: mn condition N=3}~(b).
    
    Also, since $f \mid F = y^\ell (D^{(N)})^n x$ has a $x^{n+1}y^\ell$-term, we can show that $f_x \ne 0$ by a similar argument. 

   The proof for the cases of $(D^{(N)})^n y$ or $(D^{(N)})^n z$ can be done in a similar manner.
    \end{proof}

    Now we are ready to prove the following proposition.
    
    \begin{prop}\label{thm: simple zeros N} Let $N = 1, 2, 3$. For  each integer $n\ge1$, we have 
    \begin{align*} \ngcd\left((D^{(N)})^{n} x, (D^{(N)})^{n+1} x \right)&= \ngcd\left(((D^{(N)})^{n}y, (D^{(N)})^{n+1}y\right) \\
    & = \ngcd\left(((D^{(N)})^{n}z, (D^{(N)})^{n+1}z\right) =1.\end{align*} Here $\ngcd(f, g)$ is the gcd of the numerators of reduced forms of $f$ and $g$, where $f, g  \in \QQ[x, y, z, 1/y]$.
    \end{prop}
    \begin{proof} 
    The key idea for the proof is
    to use the interplaying properties such as Lemma~\ref{lem: interplaying N} and Lemma~\ref{lem: integral N}, 
    and to consider $\px \left( (D^{(N)})^{n+1}x\right) $ and $(D^{(N)}) \left( (D^{(N)})^n x \right)$.
    
    Assume that a prime polynomial $f\in \QQ[x,y, z]$ is a common factor of the numerators of reduced forms of $(D^{(N)})^{n} x$ and $(D^{(N)})^{n+1} x$. 
    Then we write\[ (D^{(N)})^n x = f^k g/y^{\ell}, \quad (D^{(N)})^{n+1}x = f^k h/y^{\ell'}\] 
    for some integers $k\ge1$ and $\ell, \ell'\ge0$, and for some $g, h \in \QQ[x,y,z]$ such that $f \nmid \gcd(g,h)$. 
    
    Note that $\ell$ and $\ell' \in \{\ell, \ell+1\}$ are integers $\ge0$. Referring to Lemma~\ref{lem: nondividing N}, we see that $y \nmid f$. 

    Referring to Lemma~\ref{lem: integral N}\eqref{eqn: interplay x power N}, we have \[ \frac{\partial}{\partial x} (D^{(N)})^{n+1}x = \left(f^k h_x + k f^{k-1}f_x h\right) /y^{\ell'} = 
    \frac{(n+1)(n+2)}{d} f^k g / y^{\ell},\] so 
    \[\left( {\textstyle\frac{(n+1)(n+2)}{d}} g  y^{\ell'-\ell} -h_x  \right)f = k f_x h  \]
    in $\QQ[x,y,z]$, thus $f\mid h$, since $f$ is a prime and $f_x \ne 0$ by Lemma~\ref{lem: nondividing N}.

    Now we consider 
    \begin{align*}
        (D^{(N)})^{n+1}x &= f^k h /y^{\ell'} = D^{(N)}(f^k \cdot gy^{-\ell}) = kf^{k-1}gy^{-\ell}D^{(N)}(f) + f^k D^{(N)}(gy^{-\ell}),
    \end{align*}
    i.e.,
    \[ k g y^{\ell'-\ell} D^{(N)}(f) = f h - f y^{\ell'}D^{(N)}(g y^{-\ell}).
    \]
    Note that $y^{\ell' - \ell}D^{(N)}(f)$ or $D^{(N)}(g y^{-\ell})$ might have some powers of $y$ in the denominators of their reduced forms, so we let $y^{\ell''}$ be the largest power dividing their denominators. Then, we have 
    \[ k g y^{(\ell''+\ell'-\ell)} D^{(N)}(f) = \left(h y^{\ell''} -  y^{\ell'' + \ell'}D^{(N)}(g y^{-\ell})\right) \cdot f
    \]
    in $\QQ[x,y,z]$. Thus, 
    \[ f \mid g y^{\ell'' + \ell' - \ell} D(f).\]
    Note that $f$ does not divide the denominator of a reduced form of $D(f)$ by Lemma~\ref{lem: nondividing N}, and also $f \nmid y $, therefore $f \mid g$. This contradicts that $f \nmid \gcd (g, h)$. This proves  $\ngcd\left((D^{(N)})^{n} x, (D^{(N)})^{n+1} x \right)=1$.

    We apply the same argument with respect to $y$ and $z$ to complete the proof.
\end{proof}

\ 

We have proved all ingredients to prove Theorem~\ref{thm: simple zero quasimodular}.  

\begin{proof}[Proof of Theorem~\ref{thm: simple zero quasimodular}]
    It follows from Theorem~\ref{gcd} and Proposition~\ref{thm: simple zeros N}.
\end{proof}

\appendix

\section{
\texorpdfstring{
Degree analysis on $\left(D^{(3)}\right)^n$: Proofs of Lemma~\ref{lem: mn condition N=3} and Lemma~\ref{lem: y powers N=3}}{Degree analysis on (D⁽³⁾)ⁿ: Proofs of Lemma 4.3 and Lemma 4.4
}}
\label{sec: appendix}

In this appendix, we give an analysis on the degrees of denominators of $\left(\mathcal{D}^{(3)}\right)^n$ at $x$, $y$, $z$, and we prove Lemma~\ref{lem: mn condition N=3} and Lemma~\ref{lem: y powers N=3},  which have been postponed from Subsection~\ref{section: degree analysis} since they can be obtained in a similar manner as done for $\left(\mathcal{D}^{(2)}\right)^n$ in Section~\ref{section: degree analysis}. We would like to note that  the case for $\left(\mathcal{D}^{(3)}\right)^n$ is more complicated for us to deal with as you can see below.

In order to prove Lemma~\ref{lem: mn condition N=3} and Lemma~\ref{lem: y powers N=3}, we will give some conditions for the monomials appearing in each of $(D^{(3)})^n x$, $(D^{(3)})^n y$, $(D^{(3)})^n z$, and show that each of the minimal degrees with respect to $y$ among the monomials of $\left(D^{(3)}\right)^n(x)$, $\left(D^{(3)}\right)^n(y)$, and $\left(D^{(3)}\right)^n(z)$ decreases by $3$ as $n$ increases by $6$, for $n\ge0$.

Let $\mathcal{D}^{(3)} := dD^{(3)}= 6 D^{(3)}$ so that they have the integral coefficients. Then, 
\[ \mathcal{D}^{(3)}  = (x^2 - y) \frac{\partial}{\partial x} + (4xy - 4z) \frac{\partial}{\partial y} + (6xz - 3y^2 - 3 y^{-1}z^2 ) \frac{\partial}{\partial z},
\]
and
\begin{align*}
    \mathcal{D}^{(3)}(x^a y^b z^c) =(a+4b+6c)x^{a+1}y^b z^{c}-3c x^a y^{b+2} z^{c-1}  - (4b+3c) x^a y^{b-1}z^{c+1}  -a x^{a-1}y^{b+1} z^c. 
\end{align*}

We choose the same $T$ as defined in \eqref{eq:T}, and let  an integral lattice point $(\lambda, \nu)$ represent the monomial $x^a y^b z^c$ via relation \eqref{eqn: monomial to plane} as it is done for $N=2$ case, i.e., 
\[(a, b, c)^\tp = (1,r/2,0)^\tp +T^{-1} (\lambda, \nu, 0)^\tp.\] 

\begin{proof}[Proof of Lemma~\ref{lem: mn condition N=3}]
    Let $x^a y^b z^c$ be a monomial in $\left(\mathcal{D}^{(3)}\right)^r x$.  If 
    $\mathcal{D}^{(3)} (x^a y^b z^c) $  contains a monomial whose $\deg_y$ is less than 
    $b$, then such a monomial is $-(4b+3c)x^a y^{b-1}z^{c+1}$, 
    which is non-zero only when $4b+3c \ne 0$. Since the condition $4b+3c=0$ 
    is equivalent to $\nu  = \frac{2}{3}(\lambda -r)$ via 
    \eqref{eqn: monomial to plane}, the points on this line doesn't 
    produce a monomial with decreased $\deg_y$ through the 
    differential operator.

    (a) holds for the initial case $r=0$. 
    As $r$ increases, all points keep lying above or on the line $\nu  = \frac23 (\lambda -r)$; so we have $\nu \ge \frac{2}{3}(\lambda -r)$ for all $r$. This proves (a).
    
    The inequality from (a) is equivalent to $4b+3c \ge0$ by \eqref{eqn: monomial to plane}, and we conclude (b) since $4b+3c\ge0$, $a\ge0$, and $a+2b+3c=r+1$.
\end{proof}

\begin{proof}[Proof of Lemma~\ref{lem: y powers N=3}]
Let  $a_{i,j,k}^{(r)}$ denote the coefficient of $x^i y^j z^k$ in $\left( \mathcal{D}^{(3)}\right)^{r}x$. 


 We prove the lemma by showing that the minimal value among $\deg_y$ of the monomials in $\left(\mathcal{D}^{(3)}\right)^r(x)$
decreases exactly when applying $\mathcal{D}^{(3)}$ to each of $\left(\mathcal{D}^{(3)}\right)^{(6k+2)}(x)$, $\left(\mathcal{D}^{(3)}\right)^{(6k+3)}(x)$ and $\left(\mathcal{D}^{(3)}\right)^{(6k+4)}(x)$. It is enough to show that the coefficients $a_{0, -3k-1, 4k+2}^{(6k+3)}$, $a_{0, -3k, 4k+1}^{(6k+2)}$, $a_{0, -3k-2, 4k+3}^{(6k+4)}$ never vanish for $k\ge0$. We show by the non-triviality of those coefficients modulo~$7$.  

By some tedious calculations, we get;
\begin{alignat*}{2}
    A := a_{0, -3k-1, 4k+2}^{(6k+3)} & =-3 a_{0, -3k, 4k+1}^{(6k+2)}, \text{ and } a_{0, -3k, 4k+1}^{(6k+2)} & = -4 a_{0, -3k+1, 4k}^{(6k+1)},
    \end{alignat*}
therefore,  $A  = 12a_{0, -3k+1, 4k}^{(6k+1)}$. Similarly, since $a_{0, -3k+1, 4k}^{(6k+1)}  = - 5 a_{0, -3k+2, 4k-1}^{(6k)} -  a_{1, -3k, 4k}^{(6k)}$, we have \[ A  = -12\left( 5 a_{0, -3k+2, 4k-1}^{(6k)} +  a_{1, -3k, 4k}^{(6k)}\right).\]
Since  $\begin{cases} a_{0, -3k+2, 4k-1}^{(6k)} &= -6 a_{0, -3k+3, 4k-2}^{(6k-1)} - 12k a_{0, -3k, 4k}^{(6k-1)}  - a_{1, -3k+1,4k-1}^{(6k-1)}, \\
a_{1, -3k, 4k}^{(6k)} &= 12k a_{0, -3k, 4k}^{(6k-1)} - a_{1, -3k+1, 4k-1}^{(6k-1)},
\end{cases}\quad $
we have \[ A = 12 \times 6 \left( 5 a_{0, -3k+3, 4k-2}^{(6k-1)} + 
8k a_{0,-3k, 4k}^{(6k-1)} + a_{1, -3k+1, 4k-1}^{(6k-1)} \right). \]
Since 
$ \begin{cases} a_{0, -3k+3, 4k-2}^{(6k-1)} &= -7 a_{0, -3k+4, 4k-3}^{(6k-2)} + 
(-12k+3)a_{0, -3k+1, 4k-1}^{(6k-2)} 
- a_{1, -3k+2, 4k-2}^{(6k-2)}, \\
a_{0,-3k, 4k}^{(6k-1)}  &= -a_{0, -3k+1, 4k-1}^{(6k-2)},\\
 a_{1, -3k+1, 4k-1}^{(6k-1)}&= (12k-2)a_{0, -3k+1, 4k-1}^{(6k-2)} -2 a_{1, -3k+2, 4k-2}^{(6k-2)},\end{cases}\quad $ 
we have 
\begin{align*} A &= 12 \times 6 \left(-35 a_{0, -3k+4, 4k-3}^{(6k-2)} + 
(-56k+13)a_{0, -3k+1, 4k-1}^{(6k-2)}  
-7 a_{1, -3k+2, 4k-2}^{(6k-2)} \right) \equiv -2 a_{0, -3k+1, 4k-1}^{(6k-2)} \pmod{7}. \end{align*}
Finally, since $a_{0, -3k+1, 4k-1}^{(6k-2)} = -2 a_{0, -3k+2, 4k-2}^{(6k-3)} $, we also have 
\begin{align} A = a_{0, -3k-1, 4k+2}^{(6k+3)} \equiv 4 a_{0, -3(k-1)-1, 4(k-1)+2}^{(6(k-1)+3)}\pmod{7}. \label{eqn: N=3 deg minor 1}\end{align}
Recalling that  $\mathcal{D}^{(3)}(x) = 6 x^4-36 x^2 y+48 x z-6 y^2-\frac{12 z^2}{y}$, 
 since $a_{0, -1, 2}^{(3)}= -12 \not\equiv 0\pmod 7$, the coefficient $a_{0, -3k-1, 4k+2}^{(6k+3)}$ never vanishes. 
Many parts of the calculations for the recurrences of the coefficients are recyclable for other two cases. In short, we get
\begin{align*} 
    a_{0, -3k, 4k+1}^{(6k+2)} & = -4 a_{0, -3k+1, 4k}^{(6k+1)} = 4 \left(5 a_{0, -3k+2, 4k-1}^{(6k)} +  a_{1, -3k, 4k}^{(6k)}\right)\\
    & = -24 \left( 5 a_{0, -3k+3, 4k-2}^{(6k-1)} + 8k a_{0, -3k, 4k}^{(6k-1)} + a_{1, -3k+1, 4k-1}^{(6k-1)} \right)\\
    & = -24 \left( -35 a_{0, -3k+4, 4k-3}^{(6k-2)} + (-56k+13)a_{0, -3k+1, 4k-1}^{(6k-2)} - 7 a_{1, -3k+2, 4k-2}^{(6k-2)} \right)\\
    & \equiv 3 a_{0, -3k+1, 4k-1}^{(6k-2)} \equiv a_{0, -3k+2, 4k-2}^{(6k-3)} \equiv 4 a_{0, -3k+3, 4k-3}^{(6k-4)} \pmod{7}, \quad \text{ and }\\
    a_{0, -3k-2, 4k+3}^{(6k+4)} & = -2 a_{0, -3k-1, 4k+2}^{(6k+3)} =  6 a_{0, -3k, 4k+1}^{(6k+2)}= -24 a_{0, -3k+1, 4k}^{(6k+1)}\\
    & = 24 \left(5 a_{0, -3k+2, 4k-1}^{(6k)} +  a_{1, -3k, 4k}^{(6k)}\right)\\
    & =  -24\times 6 \left( 5 a_{0, -3k+3, 4k-2}^{(6k-1)} + 8k a_{0, -3k, 4k}^{(6k-1)} + a_{1, -3k+1, 4k-1}^{(6k-1)} \right)\\
    & \equiv 4 a_{0, -3k+1, 4k-1}^{(6k-2)} \pmod{7},
\end{align*}
thus 
\begin{align} \label{eqn: N=3 deg minor 2}
&  a_{0, -3k, 4k+1}^{(6k+2)} \equiv 4a_{0, -3(k-1),4(k-1)+1}^{(6(k-1)+2)} \pmod 7, \\ \label{eqn: N=3 deg minor 3}
 & a_{0, -3k-2, 4k+3}^{(6k+4)}  \equiv 4 a_{0, -3(k-1)-2, 4(k-1)+3}^{(6(k-1)+4)}\pmod{7}. \end{align}
Since $a_{0, 0, 1}^{(2)}= 4 \not\equiv 0\pmod 7$ and $a_{0, -2, 3}^{(4)}= 24 \not\equiv 0\pmod 7$, we conclude that the coefficients $a_{0, -3k-1, 4k+2}^{(6k+3)}$, $a_{0, -3k, 4k+1}^{(6k+2)}$, and $a_{0, -3k-2, 4k+3}^{(6k+4)}$ never vanish for $k\ge0$, i.e., they provide non-zero monomials with decreased $\deg_y$ after  applying $\mathcal{D}^{(3)}$. With Lemma~\ref{lem: mn condition N=3}\ref{lem: sharp ineq N=3}, 
the minimal $\deg_y$ among monomials appearing in $\left(\mathcal{D}^{(3)}\right)^{r+1}(x)$ drops by $1$ than those for $\left(\mathcal{D}^{(3)}\right)^{r}(x)$, only when $r\equiv 2,3, 4\pmod{6}$
as shown in \eqref{eqn: N=3 deg minor 1}, \eqref{eqn: N=3 deg minor 2} and \eqref{eqn: N=3 deg minor 3}. Therefore we conclude that 
$y^{\ell_r}\cdot \left(D^{(3)}\right)^r x \in \QQ[x,y,z] \setminus y\QQ[x,y,z]$, with $\ell_1, \dots, \ell_6 = 0, 0, 1, 2, 3, 3$ and $\ell_{r+6} = \ell_r + 3$ for all $r\ge1$. 

The analyses for $\left( D^{(3)}\right)^r y$ and $\left( D^{(3)}\right)^r z$ can be done in the same way, and this completes the proof.
\end{proof} 

\section*{Acknowledgement}
We appreciate  Professor SoYoung Choi for suggesting this problem.

\bibliographystyle{plainnat}

\end{document}